\documentclass[12pt,reqno]{amsart}
\usepackage[utf8]{inputenc}
\usepackage[margin=1in]{geometry}
\usepackage{amsmath}
\usepackage{amssymb,latexsym}
\usepackage{enumerate}
\usepackage{cite}
\usepackage{xcolor}

\allowdisplaybreaks[4]

\newtheorem{thm}{Theorem}[section]

\newtheorem{lem}[thm]{Lemma}
\newtheorem{prop}[thm]{Proposition}

\newtheorem*{hypothesis-s}{Hypothesis S}

\theoremstyle{remark}
\newtheorem{rem}{Remark}[section]

\def \P  {{\mathcal P}}
\def \NN {\mathbb N}
\def \CC {\mathbb C}
\def \QQ {\mathbb Q}
\def \RR {\mathbb R}
\def \ZZ {\mathbb Z}
\def \KK {\mathbb K}

\def \F  {{\mathcal F}}

\def \I  {{\mathcal I}}
\def \J  {{\mathcal J}}

\def \P  {{\mathcal P}}

\def \M  {{\mathcal M}}

\def \R  {{\mathcal R}}

\def \d {\text{d}}
\renewcommand{\Re}{\mathop{\mathrm{Re}}}

\def \sgn {\text{sgn}}

\newcommand{\mop}{(\text{mod}\,p)}

\newcommand{\mopq}{(\text{mod}\,pq)}
\newcommand{\moq}{({\rm mod}\,q)}

\newcommand{\modd}{(\text{mod}\,d)}
\newcommand{\modq}{(\text{mod}\,q_1)}
\newcommand{\moddq}{(\text{mod}\,dq_1)}
\newcommand{\motq}{(\text{mod}\,q_2)}
\newcommand{\moddelta}{(\text{mod}\,\delta)}

\numberwithin{equation}{section}

\begin{document}
\vskip 5mm

\title[]{\textbf {Additive divisor problem for multiplicative functions}}

\author{Yujiao Jiang}

\address{Yujiao Jiang\\
School of Mathematics and Statistics
\\
Shandong University
\\
Weihai
\\
Shandong 264209
\\
China}
\email{yujiaoj@sdu.com.cn}

\author{Guangshi L\"u}

\address{Guangshi L\"u\\
School of Mathematics
\\
Shandong University
\\
Jinan
\\
Shandong 250100
\\
China}
\email{gslv@sdu.edu.cn}

\date{\today}

\begin{abstract}
\small {Let $\tau$ denote the divisor function, and $f$ be any multiplicative function that satisfies some mild hypotheses. We establish the asymptotic formula or non-trivial upper bound for the shifted convolution sum $\sum_{n \leq X}f(n)\tau(n-1)$. We also derive several applications to multiplicative functions in the automorphic context, including the functions $\lambda_{\pi}(n), \,\mu(n)\lambda_{\pi}(n)$ and $\lambda_{\phi}(n)^l$. Here $\lambda_{\pi}(n)$ denotes the $n$-th Dirichlet coefficient of $\text{GL}_m$ automorphic $L$-function $L(s,\pi)$ for an automorphic irreducible cuspidal representation $\pi$, $\lambda_{\phi}(n)$ denotes the $n$-th Fourier coefficient of a holomorphic or Maass cusp form $\phi$ on ${\rm SL}_2(\mathbb Z)$, and $\mu(n)$ denotes the M\"obius function.
	
	We present two different arguments. The first one mainly relies on the uniform estimates for the binary additive divisor problem, while the second is based on the recent estimates of Bettin--Chandee for trilinear forms in Kloosterman fractions.  In addition, the Bourgain--K\'atai--Sarnak--Ziegler criterion and Linnik's dispersion method are both employed in these two arguments.
}
\end{abstract}

\subjclass[2010]{11N37, 11N36, 11P55, 1F30, 11F66}
\keywords{multiplicative functions, shifted convolution, sieve method, dispersion method, automorphic $L$-functions}
\maketitle

\section{Introduction}

Let $f(n)$ be a multiplicative function, and let $\tau(n)$ be the number of divisors of the integer $n$. We shall consider the problem to find the asymptotic behavior of the sum
\[
\sum_{n\leq X} f(n) \tau(n-1).
\]

One important example is the case of $f(n) = \tau_k(n)$, which is the number of representations of $n$ as a product of $k$ factors. This corresponding problem is well known in number theory and is called the additive divisor problem. Its evaluation has a long and rich history.  For $k=2$,  the sum was first estimated by Ingham, and the current best estimate that we know is
\[
\sum_{n\leq X}\tau(n) \tau(n-1)=XP_{2}(\log X)+O\big(X^{\frac{2}{3}+\varepsilon}\big),
\]
where $P_{k}(x)$ is a polynomial of degree $k$ (see \cite{Motohashi-1994} for more details). The case $k=3$ was first settled by Hooley $[15]$, whose method depends on the distribution of $\tau(n)$ in arithmetic progressions and an identity for $\tau_3(n)$.  Later, Heath-Brown \cite{Heath-Brown-1986} gave an asymptotic formula with a power-saving error term. Recently, Topacogullari \cite{Topacogullari-2016} subsequently improved their results and showed that
\[
\sum_{n\leq X}\tau_3(n) \tau(n-1)=XP_{3}(\log X)+O\big(X^{\frac{8}{9}+\varepsilon}\big).
\]
The first result for general $k$ was achieved by Linnik \cite{Linnik-1963},  who gave the asymptotic formula
\[
\sum_{n\leq X} \tau_k(n) \tau(n-1)=X P_{k}(\log X)+E_k(X)
\]
with $E_k(X)\ll X(\log X)^{k-1}(\log\log X)^4$ by appealing to his own powerful dispersion method. This result was improved by Motohashi \cite{Motohashi-1980} to obtain $E_k(X)\ll X(\log \log X)^{c}(\log X)^{-1}$ for some constant $c=c(k)$. Here Motohashi applied a combination of uniform distribution of $\tau_k(n)$ in arithmetic progressions, Linnik's dispersion method, and the Hardy--Littlewood circle method. Further, Fouvry--Tenenbaum \cite{FoTen-1985}, Drappeau \cite{Drappeau-2017} and Topacogullari \cite{Topacogullari-2018} investigated the cancellations among the average of Kloosterman sums or Kloosterman fractions and made successive improvements on the error term $E_k(X)$.

In addition, various generalizations of the additive divisor problem were also studied extensively by many number theorists. These include some specific multiplicative functions such as $f(n)=\tau_k^m(n)$, and some classes of multiplicative functions which satisfy specific conditions.  In 1973, Wolke \cite{Wolke-1973} showed the following result: assume the multiplicative function $f$ satisfies
\[
\left|f\left(p^{a}\right)\right| \leq D_{1} a^{D_{2}} \quad \text { for all primes } p \text { and integers } a\geq 1,
\]
and
\[
\sum_{p \leq X}|f(p)-b| \ll \frac{X}{(\log X)^A} \text { for all } A>0,
\]
where $b$ is a complex number and $D_{1}, D_{2}$ are positive real constants. Then one has
\begin{equation}\label{eq-wolke}
\sum_{n\leq X} f(n) \tau(n-1)= C_f X (\log X)^b+O\big(X (\log X)^{\Re b -1}(\log \log X)^c\big),
\end{equation}
where the constants $C_f $ and $c$ depend on $f$. We also recall an interesting result \cite{TT-1998} of Timofeev and Tulyaganov:  for any non-negative  multiplicative function $f(n)$ satisfying (i) $f\left(p^{r}\right) \leq A^{r}$ for some $A>0$, (ii) $f(n) \ll n^{\varepsilon}$ for any $\varepsilon>0$, (iii) $\sum_{p \leq X} f(p) \ln p \geq \alpha X$ with some $\alpha>0$, then  one has
\begin{equation}\label{eq-TT}
\sum_{n \leq X} f(n) \tau(n-1)=(C_f +o(1))\sum_{n \leq X} f(n) \log X
\end{equation}
with an explicitly given constant $C_f$.  What is more, Drappeau and Topacogullari \cite{DT-2019} consider another class of multiplicative functions which satisfy (i) $f(p_{1})=f(p_{2})$ for any primes $p_{1}$ and $p_{2}$ with $p_{1} \equiv p_{2} (\text{mod} \, D)$, (ii) $|f(n)| \leq \tau_{k}(n)$ for some $k$, where $D$ is some fixed integer. They derived an asymptotic formula with a better error term
\begin{equation}\label{eq-DT}
\sum_{n \leq X} f(n) \tau(n-1)=2 \sum_{\substack{\chi \text { primitive } \\ \operatorname{cond}(\chi) \mid D}} \sum_{\substack{q \leq \sqrt{X} \\ \operatorname{cond}(\chi) \mid q}} \frac{1}{\varphi(q)} \sum_{\substack{q^{2} \leq n \leq X \\(n, q)=1}} f(n) \chi(n)+O\Big(\frac{X}{(\log X)^{A}}\Big),
\end{equation}
where the implied constant depends only on $k, D$ and $A$. They also provided  some interesting applications for $f$ being the arithmetic functions $\tau_z(n), b_{\KK}(n), z^{\omega(n)}$, where $z\in\CC$, $\omega(n)$ is the number of distinct prime divisors of an integer $n$, and $\tau_z(n), b_{\KK}(n)$ are Dirichlet coefficients of $\zeta^z(s), \zeta^{1/[\KK:\QQ]}(s)H(s)$ with some Euler product $H(s)$. Finally, it is mentioned that Fouvry and Tenenbaum \cite{FoTen-2022} derives these applications from an alternative argument.

We find that the previously considered multiplicative functions $f$ should at least satisfy the condition $f(n) \ll n^{\varepsilon}$. It is of interest to know whether one can derive an asymptotic formula for sum $\sum_{n \leq X} f(n) \tau(n-1)$ without the restriction $f(n) \ll n^{\varepsilon}$. We further expect that our result can apply to coefficients of $\text{GL}_m$ automorphic $L$-functions for any fixed $m$. Consequently, the purpose of this paper is to extend the previous results for a broader class of multiplicative functions. Let $\F$ denote the class of all multiplicative functions  $f$ with the following hypotheses:

\smallskip
(i) The second moment of $f$ is bounded by
\begin{equation*}
\sum_{n \leq X}|f(n)|^2 \ll X(\log X)^{c-1}
\end{equation*}
for some constant $c\geq 1$.
\smallskip

(ii)   Let $\varepsilon>0$ denote an arbitrary small constant, and $P$ denote the product of primes $p$ which belong to the interval $[\exp\big((\log X)^{\varepsilon/2}\big), \, \exp\big((\log X)^{1-\varepsilon/2}\big)]$. Then $f$ satisfies the following upper bound estimate condition in sieve theory
\begin{equation*}
\sum_{\substack{n \leq X\\ \left(n, P\right)=1}}|f(n)|^2 \ll_{\varepsilon} \frac{X}{(\log X)^{1-\varepsilon}}.\\
\end{equation*}

(iii)  The function $f$ satisfies the Siegel--Walfisz criterion at primes, which says that for any fixed $A>0$,
\[
\sum_{\substack{p\leq X\\ p\equiv a\moq}} f(p)-\frac{1}{\varphi(q)}\sum_{p\leq X} f(p) \ll_A\frac{X}{(\log X)^A}
\]
holds for all  $(a,q)=1$ .

\begin{thm}\label{thm-main}
	Let $f\in \F$. Then we have
	$$
	\sum_{n \leq X}f(n)\tau(n-1)=\sum_{n \leq X}f(n) \big(c(0,n)(\log n+2\gamma)+2c^{\prime}(0, n)\big)+O\big(X(\log X)^{\frac{1}{2} + \varepsilon}\big),
	$$
	where $c(s,  n)$ is given by
	\[
	c(s, n)=\prod_{p \mid n}\left(1-\frac{1}{p^{s+1}}\right) \prod_{p \nmid n}\left(1+\frac{1}{(p-1) p^{s+1}}\right),
	\]
	and the implied constant depends on $\varepsilon$ and the implied constants in the hypotheses.
\end{thm}

\begin{rem}\label{rem-bettererror}
	If there exists some constant $k$ such that $|f(n)| \leq \tau_{k}(n)$ for all integers $n$, then the error term can be improved to  $O\big(X(\log X)^{ \varepsilon}\big)$. This can be seen from Remark \ref{rem-bv} or Remark \ref{rem-J}.
\end{rem}

\begin{rem}\label{rem-Mcomput}
	The main term in Theorem \ref{thm-main} can be evaluated by some standard methods in complex analysis, for instance Perron's formula. First of all, by direct calculation, one has
	\[
	c(s,  n)=\prod_{p}\left(1+\frac{1}{(p-1) p^{s+1}}\right) \,\prod_{p \mid n}\left(1-\frac{p}{(p-1) p^{s+1}+1}\right):=h(s) g_{s}(n).
	\]
	It is worth pointing out that $f(n) g_{s}(n)$ is multiplicative in variable $n$. Consider the Dirichlet series
	\[
	D(w,s)=\sum_{n=1}^\infty f(n) g_{s}(n)n^{-w}
	\]
	for $\Re w>1$ and $\Re s>-1$. Then we may obtain the asymptotic formula for
	\[
	\sum_{n \leq X}f(n)  g_{s}(n)
	\]
	when $\Re s>-1$. Finally, the main term follows from partial summation and taking the derivative on variable $s$. The detailed process will be seen in the applications.
\end{rem}

\begin{rem}
	This should be compared with the asymptotic formulae \eqref{eq-wolke}, \eqref{eq-TT} and \eqref{eq-DT} of Wolke, Timofeev--Tulyaganov and  Drappeau--Topacogullari, respectively. It is known that the error term in the work of Drappeau--Topacogullari is  $\ll X(\log X)^{-A}$, which is stronger than our result. Note that Wolke's result is better than ours only if $\Re b\leq 3/2$. Since the multiplicative functions that they consider are essentially periodic on the set of primes or close to a fixed number at primes on average, it is reasonable that these functions have better asymptotic behaviors. Moreover, our result is noticeably better than that of Timofeev--Tulyaganov.
\end{rem}

Finally, we provide a brief overview of the proof of Theorem \ref{thm-main}. We will give two different arguments, and they are both based on the generalized Bourgain--K\'atai--Sarnak--Ziegler (BKSZ for short) criterion and Linnik's dispersion method. The BKSZ  criterion is actually a finite version of Vinogradov's bilinear technique, which appeared in the works \cite{BSZ-2013, Katai-1986}.  It is very effective and widely applied to attack Sarnak's disjointness conjecture. For our purpose, we have to generalize the BKSZ criterion to a greater extent, when compared with the works \cite{CafPerZac20, JL2019} of Cafferata--Perelli--Zaccagnini and Jiang--L\"u. It is worth illustrating that these two arguments have completely different starting points and key ingredients. We shall explain them below, respectively.

Our first approach is to use the idea of BKSZ at the beginning. Then the sum that we are concerned with
\[
\sum_{n\leq X} f(n) \tau(n-1)
\]
can be rearranged as a bilinear sum
\begin{equation}\label{eq-bilinear}
\sum_{m\leq X/N} \,\sum_{p\sim N} f(m)f(p)\tau(pm-1)
\end{equation}
and a sifted sum
\begin{equation*}\label{eq-seive}
\sum_{\substack{n\leq X \\  (n,P)=1}} f(n) \tau(n-1),
\end{equation*}
where $\exp\big((\log X)^{\varepsilon/2}\big)\ll N\ll \exp\big((\log X)^{1-\varepsilon/2}\big)$.  After applying the Cauchy--Schwarz inequality, the bilinear sum is generally reduced to estimate
\[
\sum_{p_1,p_2\sim N} f(p_1)f(p_2) \sum_{m\leq X/N} \tau(p_1m-1)\tau(p_2m-1).
\]
Notice that the divisor function $\tau$ is always positive. Though the innermost sum is the binary divisor problem, which has an asymptotic formula uniformly in parameters $p_1,p_2$, we can not precisely calculate the main term of \eqref{eq-bilinear} in this way, due the application of Cauchy's inequality. In order to overcome the positivity of $\tau,$ one may construct the approximation $\widetilde{\tau}$ of $\tau$ such that the difference of
$\sum_{n\leq X} f(n) \tau(n-1)$ and $\sum_{n\leq X} f(n) \widetilde{\tau}(n-1)$ is admissible, where $\widetilde{\tau}$ is also multiplicative. This may work if one refers to the idea in \cite[Lemma 3.3]{MRT-2019}. Building on the works of Linnik \cite[Chapter II(2)]{Linnik-1963} and Motohashi \cite{Motohashi-1980}, we here choose another approximation of $\tau(n-1)$ by $\tau(n+h)$, where $h$ is any prime in $[X^{2/3},  X(\log X)^{-A}]$. The reason for our choice will be explained in Section 5. If the difference
\begin{equation}\label{eq-intro-diff}
\sum_{n\leq X} f(n) \big(\tau(n-1)-\tau(n+h)\big)
\end{equation}
is admissible for any prime $h\in [X^{2/3},  X(\log X)^{-A}]$, then the origin additive divisor problem is reduced to a ternary additive problem
\[
\sum_{h\leq X(\log X)^{-A}}\sum_{n \leq X} \Lambda(h)f(n)\tau(n+h),
\]
which can be solved easily by the Hardy--Litterwood circle method. Notice that the interval of $h$ is large enough. So we may study the summation over $h$
\[
\sum_{h\leq X(\log X)^{-A}}\Lambda(h)\tau(n+h),
\]
which is just the classical Titchmarsh problem with a large shifted parameter $n\in(0,X]$.

In order to prove that the sum \eqref{eq-intro-diff} is admissible, we take the strategy discussed as before.  We need to estimate the dispersion of difference
\begin{equation}\label{eq-bilinear-2}
\sum_{m\leq X/N} \Big|\sum_{p\sim N}f(p)\big(\tau(pm-1)-\tau(pm+h)\big)\Big|^2
\end{equation}
and the sifted sum
\begin{equation}\label{eq-seive-2}
\sum_{\substack{n\leq X \\  (n,P)=1}} f(n) \tau(n+\ell),
\end{equation}
where $\ell=-1$ or $\ell=h$ as above. The study of sum \eqref{eq-bilinear-2} then lies on the evaluation of four binary divisor sums
\begin{equation}\label{eq-binary-asy}
\sum_{m \leq X/N} \tau(p_1 m+\varpi_{1}) \tau(p_2 m+\varpi_{2}),
\end{equation}
where $p_1,p_2\sim N$, and $\left(\varpi_{1}, \varpi_{2}\right)$ is one of the pairs
$
(-1,-1),\, (-1, h),\, (h, -1),\,(h, h).
$
Moreover, the main term of \eqref{eq-binary-asy} is required to be carefully handled. We hope that the contribution of four main terms can mutually cancel when summing over prime variables $p_1,p_2$ with the weight $f(p_1)\overline{f(p_2)}$. At this moment, the Siegel--Walfisz criterion in Hypothesis (iii) plays an important role. The sifted sum \eqref{eq-seive-2} is a standard sieve condition, and is estimated by Hypothesis (ii) and Lemma \ref{lem-sieve-divisor}.

\vskip 2mm

Our second approach is based on the hyperbola method, which means that  $\tau(n)$ is written as a convolution $\sum_{ab=n}1$. Thus, the problem is reduced to investigating the problem of Bombieri--Vinogradov type for $f(n)$ with a range of moduli at least $\sqrt{X}$. In fact, we can prove the following.

\begin{thm} \label{thm-BV} Let $f\in \F$. For any $\varepsilon>0$, we have
	\begin{equation*}
	\sum_{\substack{q \leq X^{\frac{17}{33} - \varepsilon}}} \Big | \sum_{\substack{n \leq X \\ n \equiv 1\moq}} f(n) - \frac{1}{\varphi(q)} \sum_{\substack{n \leq X \\ (n,q) = 1}} f(n) \Big | \ll X(\log X)^{\frac{1}{2} + \varepsilon},
	\end{equation*}
	where the implied constant depends on $\varepsilon$ and the implied constants in the hypotheses.
\end{thm}

\begin{rem}\label{rem-bv}
	Let $k$ be a fixed positive integer. Assume $f\in \F$ satisfies $|f(n)| \leq \tau_{k}(n)$ for all integers $n$. By the work \cite[Corollary]{FR-2022} of Fouvry and Radziwi\l\l,   one has
	\[
	\sum_{\substack{q \leq X^{\frac{17}{33} - \varepsilon}}} \Big | \sum_{\substack{n \leq X \\ n \equiv 1\moq}} f(n) - \frac{1}{\varphi(q)} \sum_{\substack{n \leq X \\ (n,q) = 1}} f(n) \Big | \ll_{\varepsilon}X(\log X)^{\varepsilon}.
	\]
	Hence, the restriction $|f(n)| \leq \tau_{k}(n)$ is removed at cost of the magnitude $(\log X)^{\frac{1}{2}}$. We also mention that Green \cite{Green} and Granville--Shao \cite{GS-2019} considered the multiplicative functions with stronger restriction $|f(n)|\leq 1$.
\end{rem}

\vskip 2mm

The proof of Theorem \ref{thm-BV} starts with the generalized BKSZ criterion, which can decompose the sum above into two parts. One part need to estimate
\[
\sum_{\substack{q \leq X^{\frac{17}{33} - \varepsilon}}} \sum_{\substack{n \leq X\\(n,P)=1\\ n \equiv 1\moq}} |f(n)|.
\]
If the function $f$ is divisor-bounded, the Titchmarshi--Brun inequality \cite{Shiu} can be used directly to estimate the inner sum (see Lemma 9.1 in \cite{FR-2022}). However, it is not suitable for our situation. After exchanging the order of summations,  the above double sum is transformed into the sifted sum similar to \eqref{eq-seive-2}, and can be treated in the same manner.

The other part can be rewritten as a combination of the sums of type
\[
\sum_{m\sim M}f(m)\Big( \underset{\substack{q \sim Q\ p\sim N\\ pm \equiv 1 \moq}}{\sum\  \sum } c_q f(p)  -
\underset{\substack{q \sim Q\ p\in\P_\nu  \\ (pm, q)=1}}{\sum\  \sum } \frac{c_qf(p)}{\varphi (q)}  \Big),
\]
where $\exp\big((\log X)^{\varepsilon/2}\big)\ll N\ll \exp\big((\log X)^{1-\varepsilon/2}\big)$, $X^{3/4}\ll M\ll X/N$ and $Q\ll X^{17/33- \varepsilon}$.
After applying the Cauchy--Schwarz inequality, it reduces to bounding the dispersion
\begin{equation*}
\begin{aligned}
\sum_{m\sim M}\psi \Bigl( \frac{m}{M}\Bigr)\Big| \underset{\substack{q \sim Q\ p\sim N\\ pm \equiv 1 \moq}}{\sum\  \sum } c_q f(p)  -
\underset{\substack{q \sim Q\ p\in\P_\nu  \\ (pm, q)=1}}{\sum\  \sum } \frac{c_qf(p)}{\varphi (q)}  \Big|^2,
\end{aligned}
\end{equation*}
where $\psi$ is some compactly supported smooth function. We then follow the usual steps as in Linnik's dispersion method. Opening the square makes the above sum split into three terms.  The most difficult term is to estimate
\[
\sum_{q_1\sim Q} \sum_{q_2\sim Q} c_{q_1} \overline{c_{q_2}} \sum_{\substack{p_1,p_2\sim N}}
f(p_1)\overline{f(p_2)}  \sum_{\substack{p_1m\equiv 1 \modq \\ p_2m\equiv 1 \motq}} \psi \Bigl( \frac{m}{M}\Bigr).
\]
Next, some familiar techniques and tools are used, including the Poisson summation formula, the Barban--Davenport--Halberstam theorem, Bezout's identity and sums of Kloosterman fractions. Especially, we benefit from the key estimate of Bettin--Chandee for trilinear forms in Kloosterman fractions (see Lemma \ref{lem-trilinear}).

\vskip5mm

\section{Applications}
We now provide some examples, which fit into the framework of our theorem \ref{thm-main}.

\vskip 2mm
\subsection{Automorphic forms on $\mathrm{GL}_m$}
Let $m\geq 2$, and let $\mathcal{A}(m)$ be the set of all cuspidal automorphic representations of $\mathrm{GL}_m$ over $\QQ$ with unitary central character. For each $\pi\in \mathcal{A}(m)$, the corresponding $L$-function is defined by absolutely convergent Dirichlet series as
\begin{equation*}
L(s,\pi) =\sum_{n=1}^{\infty}\frac{\lambda_{\pi}(n)}{n^s}
\end{equation*}
for $\Re s >1$. By the Rankin--Selberg theory and the inequality $|\lambda_{\pi}(n)|^2\leq \lambda_{\pi\times \tilde{\pi}}(n)$ for all positive integers $n$ (see \cite[Lemma 3.1]{JLW}), one has
\begin{equation}\label{eq-rslambda}
	\sum_{n\leq X}|\lambda_{\pi}(n)|^2\leq \sum_{n\leq X} \lambda_{\pi\times \tilde{\pi}}(n)\ll_{\pi} X.
\end{equation}
Furthermore, by utilizing the sieve technique, we have
\[
\sum_{\substack{n\leq X \\  \left(n, P(Y,Z)\right)=1}}|\lambda_{\pi}(n)|^2\ll_{\pi} X\frac{\log Y}{\log Z}
\]
for any $N_{\pi}<Y< Z \leq X^{\frac{1}{30m^2}}$, which has been shown in \cite[Lemma 5.3]{JL-2021}. Here $N_{\pi}$ is the arithmetic conductor of $\pi$. For any fixed $\pi\in \mathcal{A}(m)$, Hypotheses (i) and (ii) then hold with $c=1$. We further assume that $\pi$ is self-dual and $\pi\not\simeq \pi\otimes\chi$ for any quadratic primitive character $\chi$. We define the subset $\mathcal{A}^\flat(m)$ to be those $\pi\in \mathcal{A}(m)$ satisfying the above assumptions.
Hypotheses (iii) directly follows from \cite[Corollary 4.8]{JLTW-2021}.

Therefore, Theorem \ref{thm-main} can apply the multiplicative functions $\lambda_{\pi}(n)$ and $\mu(n)\lambda_{\pi}(n)$. Next, using the analytic properties of the twisted automorphic $L$-functions, one may show the quantities of main terms are admissible when compared to $O\big(X(\log X)^{\frac{1}{2} + \varepsilon}\big)$. However, we want to avoid calculating the main term in Theorem \ref{thm-main} for these two cases. We will provide another approach to get the following result.

\begin{thm}\label{thm-main-2}
	Fixed $\pi\in \mathcal{A}^\flat(m)$ with $m\geq 2$. Let $\lambda_{\pi}(n)$ denote the coefficients of $L(s,\pi)$. For any $\varepsilon>0$, we have
	$$
	\sum_{n \leq X}\lambda_{\pi}(n)\tau(n-1)\ll X(\log X)^{\frac{1}{2} + \varepsilon}
	$$
	and
	$$
	\sum_{n \leq X}\mu(n)\lambda_{\pi}(n)\tau(n-1)\ll X(\log X)^{\frac{1}{2} + \varepsilon},
	$$
	where the implied constant depends on $\pi, \varepsilon$.
\end{thm}

\begin{rem}
	If $m\geq 5$, this result is completely new. The case of $m\leq 4$ has been treated in our previous work joint with Thorner and Wang \cite{JLTW-2021} through the Bombieri--Vinogradov theorem on higher rank groups.
\end{rem}

\begin{proof}
	We only prove the second assertion, since it is a little more complicated. From the relation \eqref{eq-1top}, we get
	\begin{equation}\label{eq-B-T-trans}
			\sum_{n \leq X}\mu(n)\lambda_{\pi}(n)\tau(n-1)=\frac{(\log X)^{A}}{X}T_{\mu\lambda_{\pi}}(X)+O\big(X(\log X)^{\frac{1}{2}+\varepsilon}\big),
	\end{equation}
where $T_{\mu\lambda_{\pi}}(X)$ is a ternary additive problem given by
\begin{equation*}\label{eq-defTf}
T_{\mu\lambda_{\pi}}(X)=\sum_{h\leq X(\log X)^{-A}}\sum_{n \leq X} \Lambda(h)\mu(n)\lambda_{\pi}(n)\tau(n+h).
\end{equation*}
If $g$ is finitely supported, we define the exponential sum $F_g: \RR/\ZZ \rightarrow\CC$ by the formula
\[
F_g(\alpha;X)=\sum_{n\leq X} g(n)e(n\alpha).
\]
Then we have
$$
T_{\mu\lambda_{\pi}}(X)=\int_{0}^{1} F_{\Lambda}\big(\alpha;X(\log X)^{-A}\big)F_{\mu\lambda_{\pi}}(\alpha;X)  F_{\tau}(-\alpha; 2X) \d \alpha.
$$

In order to apply the circle method, we define the major arcs and the minor arcs as follows:
\begin{align*}
	\mathfrak{M}=\bigcup_{1\leq q\leq R}\bigcup_{\substack{1\leq a\leq q\\ (a,q)=1}}\mathfrak{M}(a,q)
\end{align*}
with $\mathfrak{M}(a,q)=\big[\frac{a}{q}-\frac{1}{qQ},\ \frac{a}{q}+\frac{1}{qQ}\big]$, and $\mathfrak{m}$ is the complement of $\mathfrak{M}$ in $\mathbb{T}$:
\begin{equation*}
	\frak{m}=\mathbb{T}\setminus \frak{M}.
\end{equation*}
We take $R= (\log X)^{10}$ and $Q= X(\log X)^{-A-10}$.
Then $T_{\mu\lambda_{\pi}}(X)$ is divided into two parts as follows:
\[
\begin{aligned}
	T_{\mu\lambda_{\pi}}(X) =&\Big(\int_{\mathfrak{M}}+\int_{\mathfrak{m}}\Big)\, F_{\Lambda}\big(\alpha;X(\log X)^{-A}\big)F_{\mu\lambda_{\pi}}(\alpha;X)  F_{\tau}(-\alpha; 2X)  \d \alpha\\
	:=&T_{\mu\lambda_{\pi}}^{(1)}(X)+T_{\mu\lambda_{\pi}}^{(2)}(X).
\end{aligned}
\]
On the mirror arcs $\mathfrak{m}$, it follows from the Cauchy--Schwarz inequality and the estimate \eqref{eq-rslambda} that
\[
T_{\mu\lambda_{\pi}}^{(2)}(X)\ll \sup_{\alpha\in \mathfrak{m}}|F_{\Lambda}\big(\alpha;X(\log X)^{-A}\big)|X(\log X)^{\frac{3}{2}}.
\]
By the classical estimate of exponential sum over primes \cite[Theorme 13.6]{IK}, we have
\[
F_{\Lambda}\big(\alpha;X(\log X)^{-A}\big)\ll \Big(Q^{\frac{1}{2}}X^{\frac{1}{2}}(\log X)^{-\frac{A}{2}}+R^{-\frac{1}{2}}XX(\log X)^{-A}+X^{\frac{3}{5}}\Big)(\log X)^3\ll \frac{X}{(\log X)^{A+2}}
\]
for any $\alpha\in \mathfrak{m}$. Hence, the contribution of terms on the mirror arcs $\mathfrak{m}$ satisfies
\begin{equation}\label{eq-minor-contri}
	T_{\mu\lambda_{\pi}}^{(2)}(X)\ll \frac{X^{2}}{(\log X)^{A}}.
\end{equation}

Now we evaluate the contribution $T_{\mu\lambda_{\pi}}^{(1)}(X)$ of terms on the major arcs $\mathfrak{M}$. From now on we assume $\alpha=\beta+a/q \in \mathfrak{M}(a,q)$. It follows from \cite[Section 5]{JLW-mobius} that
\begin{equation}\label{eq-mulabmda-M}
\begin{aligned}
F_{\mu\lambda_{\pi}}(\alpha;X)\ll&\Big(1+\frac{X}{qQ}\Big)\max_{1\leq Y\leq X}\Big|F_{\mu\lambda_{\pi}}\Big(\frac{a}{q};Y\Big)\Big|\\
\ll&(\log X)^{A+10} X\exp(-c\sqrt{\log X})\\
 \ll&X\exp\big(-\frac{c}{2}\sqrt{\log X}\big),
\end{aligned}
\end{equation}
where $c>0$ is an ineffective constant due to the Siegel-type estimate of twisted automorphic $L$-functions \cite[Theorem 4.1]{JLTW-2021}. Similar to the treatment mirror arcs, we insert the bound \eqref{eq-mulabmda-M} and then obtain that the contribution of terms on $\mathfrak{M}$ satisfies
\begin{equation*}\label{eq-major-contri}
T_{\mu\lambda_{\pi}}^{(1)}(X)\ll X^2\exp\big(-\frac{c}{3}\sqrt{\log X}\big).
\end{equation*}
Combining this with \eqref{eq-minor-contri}, we get 	$T_{\mu\lambda_{\pi}}(X)\ll X^{2}(\log X)^{-A}$, which further yields from \eqref{eq-B-T-trans} that
\[
\sum_{n \leq X}\mu(n)\lambda_{\pi}(n)\tau(n-1)\ll X(\log X)^{\frac{1}{2}+\varepsilon}.
\]
This completes the proof of second assertion in Theorem \ref{thm-main-2}.
\end{proof}

In order to illustrate that Theorem \ref{thm-main-2} provides a non-trivial upper bound, we need to estimate the correct order of magnitude of
\[
\sum_{n \leq X}|\lambda_{\pi}(n)|\tau(n-1).
\]
We suppose that a cuspidal automorphic representation $\pi$ comes from certain symmetric power lift of ${\rm GL}_{2}$, for simplicity.
Let $\phi$ be any holomorphic cusp form of weight $k$ for ${\rm SL}_2(\ZZ)$, which is a normalized eigenform for the Hecke operators.
Suppose that $\lambda_{{\rm Sym}^{r} \phi}(n)$ are the Dirichlet coefficients of its $r$-th symmetric power $L$-function $L_{r}(s, \phi)$, where $r\geq 1.$ Thanks to the Sato--Tate conjecture (which is now a theorem of Barnet-Lamb, Geraghty, Harris and Taylor \cite{BGHT}), it was proved by L\"u \cite[Lemma 3.4]{lu-2015} and Tang and Wu \cite[Theorem 1]{TW-2016}
\begin{equation*}
\sum_{n\leq X}|\lambda_{{\rm Sym}^{r} \phi}(n)| \sim c_r(\phi)\frac{X}{\log^{\delta_r} X},
\end{equation*}
where $\delta_r$ can be explicitly determined with $0<\delta_r\leq 1-8/\pi^2=0.1849...$, and $c_r(\phi)$ is a positive constant depending on $\phi$ and $r.$ According to the recent work \cite{NT2019} of Newton and Thorne, there exists a cuspidal automorphic representation $\pi$ on ${\rm GL}_{r+1}$ such that
$$L(s,\pi)=L_r(s, \phi)$$
for all $r\geq 1.$ Hence, we have
\begin{equation*}
\sum_{n\leq X}|\lambda_{\pi}(n)| \gg \frac{X}{(\log X)^{1-8/\pi^2}},
\end{equation*}
which further yields from the result \eqref{eq-TT} of Timofeev and Tulyaganov that
\[
\sum_{n \leq X}|\lambda_{\pi}(n)|\tau(n-1)\gg X(\log X)^{\frac{8}{\pi^2}}.
\]
For this case, since $8/\pi^2\geq 4/5>1/2$,  it indicates that Theorem \ref{thm-main-2} does give a non-trivial saving.

\vskip 2mm

\subsection{Cusp forms on $\mathrm{SL}_2(\ZZ)$}
Let $k$ be an even positive integer, and let $H_k$ denote the set of arithmetically normalized primitive cusp forms of weight $k$ for ${\rm SL}_2(\ZZ)$ which are eigenfunctions of all the Hecke operators. Any $\phi \in H_k$ has a Fourier expansion at infinity given by
\begin{equation*}\label{fun:holo}
\phi(z)=\sum_{n=1}^\infty \lambda_{\phi}(n)n^\frac{k-1}{2}e(nz),
\end{equation*}
where $ \lambda_{\phi}(1)=1$ and the eigenvalues $\lambda_{\phi}(n)\in \mathbb{R}$. Deligne proved the Ramanujan conjecture, which asserts that
\begin{equation}\label{d-bound}
|\lambda_{\phi}(n)|\leq \tau(n)
\end{equation}
for all $n\geq1.$ Similarly, let $S_{r}$ be the set of arithmetically normalized Hecke--Maass cusp forms of eigenvalue $\lambda=1/4+r^{2}$ for ${\rm SL}_2(\ZZ)$. Then $\phi\in S_{r}$ has the Fourier expansion at infinity given by
$$
\phi(z)=\sum_{n \neq 0} \lambda_{\phi}(n) \sqrt{y} K_{i r}(2 \pi|n| y) e(n x),
$$
where $K_{i r}$ is the $K$-Bessel function and $\lambda_{\phi}(n)$ are eigenvalues of $T_{n}$ with $\lambda_{\phi}(n)=1$ and $\lambda_{\phi}(n)\in \mathbb{R}$. Up to now, the Ramanujan conjecture on Maass cusp forms has not been proved. This is different from holomorphic cusp forms. The current best result is due to Kim and Sarnak, which states that
\begin{equation}\label{eq-ksbound}
\left|\lambda_{\phi}(n)\right| \leq n^{\frac{7}{64}} \tau(n) .
\end{equation}
Moreover, for any $\phi \in H_k \cup S_{r}$, the eigenvalues $\lambda_{\phi}(n)$ enjoy the multiplicative property
\begin{equation}\label{Fc-mul-pro}
\lambda_{\phi}(m)\lambda_{\phi}(n)=\sum_{d|(m,n)}\lambda_{\phi}\Big(\frac{mn}{d^2}\Big)
\end{equation}
for all integers $m, n\geq 1.$ In particular, $\lambda_{\phi}(n)$ are multiplicative. The Hecke $L$-function $L(s,\phi)$ associated to $\phi$ has the Euler product representation
\[
L(s,{\phi})=\sum_{n\geq 1}\frac{\lambda_{\phi}(n)}{n^s}=\prod_p\left(1-\frac{\lambda_{\phi}(p)}{p^s}+\frac{1}{p^{2s}}\right)^{-1}.
\]
We rewrite the Euler product as
\[
L(s,{\phi})=\prod_p\left(1-\frac{\alpha_{\phi}(p)}{p^s}\right)^{-1}\left(1-\frac{\beta_{\phi}(p)}{p^s}\right)^{-1},
\]
where $\alpha_{\phi}(p), \beta_{\phi}(p)$ are complex numbers with $\alpha_{\phi}(p)=\overline{\beta_{\phi}(p)}, \,|\alpha_{\phi}(p)|=|\beta_{\phi}(p)|=1$. For each $m\geq 1$, we
define the $m\mbox{-}$th symmetric power $L$-function by the
degree $m+1$ Euler product
\begin{equation*}\label{L-symfun:E}
L(s,{\rm sym}^m {\phi})=\prod_p\prod_{0\leq j\leq m}\left(1-\frac{\alpha_{\phi}(p)^{m-j}\beta_{\phi}(p)^{j}}{p^s}\right)^{-1}.
\end{equation*}
Note that $L(s,{\rm sym}^0{\phi})=\zeta(s)$ and $L(s,{\rm sym}^1{\phi})=L(s,{\phi})$.

Recently, Newton and Thorne \cite[Theorem B]{NT2019} proved that if $\phi \in H_k$, then the $m$-th symmetric power lift ${\rm sym}^m \phi$ corresponds to a cuspidal automorphic representation of ${\rm GL}_{m+1}(\mathbb{A})$ with trivial central character. This result implies that ${\rm sym}^m \phi\in \mathcal{A}^\flat(m+1)$ for all $m\geq 1$.
If $\phi\in S_{r}$, it is only known that ${\rm sym}^m \phi\in \mathcal{A}^\flat(m+1)$ for $2\leq m\leq 4$, due to the works \cite{GJ-1978,KS-2002,Kim-2003} of Gelbert, Jacquet, Kim and Shahidi.

Another interesting example is the case $f(n)=\lambda_{\phi}(n)^l$, where $l$ is a positive integer. The problem that we are concerned with is the asymptotic behavior of the sum
\[
M_{\phi,l}(X)=\sum_{ n\leq X} \lambda_{\phi}(n)^l \tau(n-1).
\]
When $l=1$, the spectral method or $\delta$-symbol method and its variants can give $M_{\phi,l}(X)\ll X^{\frac{2}{3}+\varepsilon}$. Redmond \cite{RE1, RE2} considered the case of  $l=2$. However, Redmond \cite{RE3} then realized that his result for $M_{\phi,2}(X)$ was not true, and said that he did not see how to do this so far. This is to say that any non-trivial upper bound of $M_{\phi,l}(X)$ is not known if $l\geq 2$. Recently, the authors \cite{JL-2021-submit} investigated the levels of distributions of $\lambda_{\phi}(n)^2$ and $\lambda_{\phi}(n)^3$ in arithmetic progressions. When $\phi \in H_k$, with the help of Brun--Titchmarsh inequality, we derived
\[
M_{\phi,2}(X)=c_{\phi}X\log X+O\left(X\log\log X \right)
\]
and
\[
M_{\phi,3}(X)\ll X(\log X)^{\frac{64}{15\pi}-1} \log\log X,
\]
where the constant $c_\phi$ and the implied constants all depend only on $\phi$.  When $\phi \in S_r$, with the help of Large sieve inequality, we derived
\[
M_{\phi,2}(X)=c_{\phi}X\log X+O\left(X\left(\log X\right)^{\frac{1}{2}}\left(\log\log X\right)^{\frac{3}{2}} \right).
\]

In the following, we shall establish the asymptotic formulae of $M_{\phi,l}(X)$ for more cases.

\begin{thm}\label{thm-sl2power-holomorphic} Let $\phi \in H_k$ and $l\geq 2$ be any fixed integer. For any $\varepsilon>0$, we have
	\[
	M_{\phi,l}(X)=XP_{l}(\log X)+O\big(X(\log X)^{ \varepsilon}\big),
	\]
	where $P_{l}(x)$ denotes a polynomial in $x$ of degree $(2 j) ! /(j !(j+1) !)$ for even $l=2 j$, otherwise $P_{l}(x)\equiv 0$.
\end{thm}

\begin{rem}
	The conclusion of Timofeev and Tulyaganov \cite{TT-1998}, with the help of the Sato--Tate distribution, can yield
	\[
	\sum_{ n\leq X} |\lambda_{\phi}(n)|^l \tau(n-1)=(d_l(\phi)+o(1))X(\log X)^{\eta_l},
	\]
	where $\eta_l$ can be explicitly computed by
	\[
	\eta_l=\frac{2}{\pi}\int_o^{\pi}|2\cos \theta|^l (\sin \theta)^2\, \d \theta=\frac{\Gamma(l+1)}{\Gamma(l/2+2)\Gamma(l/2+1)}.
	\]
	Note that this result for even $l$ has been achieved by L\"u \cite{lu-2015}. It is obvious that our theorem \ref{thm-sl2power-holomorphic} refines the result above.
\end{rem}

\begin{proof}
	In order to apply Theorem \ref{thm-main}, we need to verify these three hypotheses. By Deligne's bound \eqref{d-bound}, we have
	\[
	\sum_{n \leq X}|\lambda_{\phi}(n)|^{2l} \leq   \sum_{n \leq X}\tau(n)^{2l} \ll X(\log X)^{4^l-1}
	\]
	which implies Hypothesis (i) holds for $c=4^l$. Inserting the bound \eqref{d-bound} again and applying the Brun--Titchmarsh inequality \cite[Theorem 1]{Shiu}, we obtain
	\begin{equation}\label{divisor-BT}
	\sum_{\substack{n\leq X \\  (n, P)=1}}|\lambda_{\phi}(n)|^{2l} \leq \sum_{\substack{n\leq X \\  (n, P)=1}}\tau(n)^{2l}\ll \frac{X}{\log X} \exp\Big(4^l \sum_{\substack{p\leq X\\ p\nmid P}}\frac{1}{p}\Big).
	\end{equation}
	Mertens' theorem can yield
	\begin{equation*}\label{sel-merten-6}
	\begin{aligned}
\sum_{\substack{p\leq X\\ p\nmid P}}\frac{1}{p}=&\sum_{p\leq X}\frac{1}{p}-\sum_{p\leq \exp((\log X)^{1-\varepsilon/2})}\frac{1}{p}+
	\sum_{p\leq \exp((\log X)^{\varepsilon/2})}\frac{1}{p}\\
	=&\varepsilon \log\log X+O(1).
	\end{aligned}
	\end{equation*}
	Inserting this into \eqref{divisor-BT} gives
	\[
	\sum_{\substack{n\leq X \\  (n, P)=1}}|\lambda_{\phi}(n)|^{2l} \ll \frac{X}{(\log X)^{1-4^l \varepsilon}}.
	\]
	which further yields Hypothesis (ii). For the last hypothesis, it is obvious from \eqref{Fc-mul-pro} that
	$\lambda_{\phi}(p)^{l}$ is a linear combination of $\lambda_{\mathrm{sym}^{j}\phi }(p)$ with $j\geq 0$. Hence, Hypothesis (iii) follows from the estimates of Sielgel--Walfisz  type for Dirichlet $L$-functions and twisted symmetric power $L$-functions (see \cite[Corollary 4.8]{JLTW-2021}).
	
	With the notation as in Theorem \ref{thm-main} and Remark \ref{rem-Mcomput}, it follows easily from Remark \ref{rem-bettererror} that
	\begin{equation}\label{eq-Ml}
	\begin{aligned}
	M_{\phi,l}(X)=&h(0) \sum_{n \leq X}\lambda_{\phi}(n)^l  g_{0}(n)\log n+2\big(\gamma h(0)+h^\prime(0)\big) \sum_{n \leq X}\lambda_{\phi}(n)^l  g_{0}(n)\\
	&+2h(0) \sum_{n \leq X}\lambda_{\phi}(n)^l  g_{0}^{\prime}(n)+O\big(X(\log X)^{\varepsilon}\big).
	\end{aligned}
	\end{equation}
	Following the idea in Remark  \ref{rem-Mcomput}, we consider the Dirichlet series
	\[
	D_{l}(w,s)=\sum_{n=1}^\infty \lambda_{\phi}(n)^l   g_{s}(n)n^{-w},
	\]
	where $g_{s}(n)$ is given by the product
	\[
	g_{s}(n)=\prod_{p \mid n}\left(1-\frac{p}{(p-1) p^{s+1}+1}\right).
	\]
	In view of the relation \eqref{Fc-mul-pro}, the series $D_{l}(w,s)$ can be decomposed as
	\begin{equation*}\label{eq-deco}
	D_{l}(w,s)=F_{l}(w) U_{l}(w,s),
	\end{equation*}
	where
	\begin{equation*}
	\begin{aligned}
	F_{2 j}(w)=\zeta(w)^{A_{j}} L\left(\mathrm{sym}^{2 j} \phi, w\right) \prod_{1 \leq r \leq j-1} L\left(\operatorname{sym}^{2 r} \phi, s\right)^{C_{j}(r)} \quad(l=2 j), \\
	F_{2 j+1}(w)=L(\phi, w)^{B_{j}} L\left(\operatorname{sym}^{2 j+1} \phi, w\right) \prod_{1 \leq r \leq j-1} L\left(\operatorname{sym}^{2 r+1} \phi, w\right)^{D_{j}(r)} \quad(l=2 j+1),
	\end{aligned}
	\end{equation*}
	and the constants $A_{j}, B_{j}, C_{j}(r), D_{j}(r)(1 \leq r \leq j-1)$ are given by
	$$
	\begin{aligned}
	A_{j} &=\frac{(2 j) !}{j !(j+1) !}, \quad B_{j}=2 \frac{(2 j+1) !}{j !(j+2) !}, \\
	C_{j}(r) &=\frac{(2 j) !(2 r+1)}{(j-r) !(j+r+1) !}, \quad D_{j}(r)=\frac{(2 j+1) !(2 r+2)}{(j-r) !(j+r+2) !} .
	\end{aligned}
	$$
	The $L$-function $F_{l}(w)$ is of degree $2^l$, and for even $l=2 j$ all coefficients of $F_{2 j}(w)$ are nonnegative. Moreover, $U_{l}(s,w)$ is a double Dirichlet series absolutely convergent in $\Re (w+s)>0$. Let us use $v_{l}(n)$ and $u_{l,s}(n)$ to denote the Dirichlet coefficients of $F_{l}(w)$ and $U_{l}(s,w)$, respectively. By Perron's formula with the analytic properties of symmetric power $L$-functions, a standard procedure yields
	\begin{equation}\label{eq-vln}
	\sum_{n\leq X}v_{l}(n)=XQ_{l}(\log X)+O(X^{1-\delta_l}),
	\end{equation}
	where $Q_{l}(x)$ denotes a polynomial in $x$, and satisfies that $Q_{2 j+1} \equiv 0$ if $l=2 j+1$ is odd while $\operatorname{deg} Q_{2 j}=(2 j) ! /(j !(j+1) !)-1$ for even $l=2 j$, and $\delta_{l}\in (0,1/3)$ is a constant depending $l$. Suppose that $\Re s>-1/2$, then we have the convolution
	$$
	\lambda_{\phi}(n)^l   g_{s}(n)=\sum_{n=ab}v_{l}(a)u_{l,s}(b)
	$$
	and
	$$
	\sum_{b=1}^\infty |u_{l,s}(b)| b^{-\sigma} \ll 1
	$$
	for any $\sigma>1/2$. With the help of \eqref{eq-vln}, we infer that
	$$
	\begin{aligned}
	\sum_{n \leq X} \lambda_{\phi}(n)^l   g_{s}(n)&=\sum_{b \leq X} u_{l,s}(b)\sum_{a \leq X/b} v_{l}(a) \\
	&=X\sum_{b=1}^\infty \frac{u_{l,s}(b)}{b} Q_{l}\left(\log \frac{X}{b}\right)+O\left(X^{1+\varepsilon} \sum_{b \geq X}|u_{l,s}(b)| b^{-1}\right)+O(X^{1-\delta_l})\\
	&=XP_{l,s}(\log X)+O(X^{1-\delta_l}),
	\end{aligned}
	$$
	where  $P_{l,s}(x)$ is a polynomial in $x$ with its coefficients depending on $s$, and satisfies that $P_{2j+1,s} \equiv 0$ if $l=2 j+1$ is odd while $\operatorname{deg} P_{2 j,s}=(2 j) ! /(j !(j+1) !)-1$ for even $l=2 j$. Subtracting this into \eqref{eq-Ml}, Theorem \ref{thm-sl2power-holomorphic} then follows.
\end{proof}

\begin{thm}\label{thm-sl2power-Maass} Let $\phi \in S_r$. For $l=3, 4$, we have
	\[
	M_{\phi,l}(X)=XP_{l}(\log X)+O\big(X(\log X)^{\frac{1}{2}+\varepsilon}\big),
	\]
	where $P_{4}(x)$ denotes a polynomial in $x$ of degree $2$, and $P_{3}(x)\equiv 0$.
\end{thm}

\begin{proof}
	As the proof of Theorem \ref{thm-sl2power-holomorphic}, we first verify three hypotheses. By the fact ${\rm sym}^m \phi\in \mathcal{A}^\flat(m+1)$ for $2\leq m\leq 4$ and the Rankin--Selberg theory, Landau's lemma gives
	\begin{equation*}
	\begin{aligned}
	\sum_{n\leq X}\lambda_{\phi}(n)^{6}=&XQ_{6}(\log  X)+O\big(X^{\frac{63}{65}+\varepsilon}\big),\\
	\sum_{n\leq X}\lambda_{\phi}(n)^{8}=&XQ_{8}(\log  X)+O\big(X^{\frac{255}{257}+\varepsilon}\big),
	\end{aligned}
	\end{equation*}
	where $Q_{6}, Q_{8}$ are polynomials with $\deg Q_{6}=4$ and $\deg Q_{8}=14$, respectively (see \cite[Remark 1.7]{LauLu-2011} for example). Thus, Hypothesis (i) holds for $\lambda_{\phi}(n)^l$ with $l=3,4$. The last hypothesis also holds from the corresponding argument of Theorem \ref{thm-sl2power-holomorphic}. Now it remains to check the second one. Since the Ramanujan conjecture is still open for Maass cusp form, the Brun--Titchamrsh inequality can not be used as in \eqref{divisor-BT}.  Fortunately, this barrier can be overcome by applying a result \cite[Theorem 01]{ht} of Hall and Tenenbaum, which states that if a non-negative multiplicative function $g(n)$ satisfies
	\[
	\sum_{p\leq X}g(p)\log p\ll X\text{  and } \quad \sum_{p\leq X}\sum_{k\geq 2}\frac{g(p^k)\log p^k}{p^k}\ll 1,
	\]
	then one has
	\begin{equation}\label{eq-HT-inequality}
	\sum_{n\leq X} g(n)\ll \frac{X}{\log X}\sum_{n\leq X} \frac{g(n)}{n}.
	\end{equation}
	We put
	\[
	g(n)=\left\{\begin{array}{ll}
	\lambda_{\phi}(n)^{8}& \text { if }(n, P)=1, \\
	0& \text { if }(n, P)>1.
	\end{array}\right.
	\]
	Clearly, $g(n)$ is multiplicative and non-negative.  The multiplicative relation \eqref{Fc-mul-pro} gives
	\begin{equation}\label{eq-identity-8}
	\lambda_{\phi}(p)^{8}=-8+8\lambda_{\phi\times \phi }(p)+7\lambda_{\rm{sym}^{2}\phi\times {\rm sym}^{2}\phi }(p)+6\lambda_{{\rm sym}^{3}\phi\times {\rm sym}^{3}\phi }(p)+\lambda_{{\rm sym}^{4}\phi\times {\rm sym}^{4}\phi }(p).
	\end{equation}
	Using Shahidi's non-vanishing result \cite{Shahidi-1981} of Rankin--Selberg $L$-functions at $\Re s=1$, we get
	\[
	\sum_{p\leq X}\lambda_{{\rm sym}^{j}\phi\times {\rm sym}^{j}\phi}(p)\log p \ll X.
	\]
	Thus, it follows from the identity  \eqref{eq-identity-8} that
	\[
	\sum_{p\leq X}\lambda_{\phi}(p)^{8}\log p \ll X.
	\]
	Moreover, by the bound \eqref{eq-ksbound} of Kim and Sarnak, we find
	\begin{equation}\label{eq-kgeq2}
	\sum_{p\leq X}\sum_{k\geq 2}\frac{\lambda_{\phi}(p^k)^{8}\log p^k}{p^k}\ll\sum_{p\leq X}\frac{(|\alpha_{\phi}(p)|^{8}+|\beta_{\phi}(p)|^{8})}{p^{1+\frac{1}{8}+\varepsilon}} \ll 1,
	\end{equation}
	where the last step uses the convexity of $L(s, {\rm sym}^{4}\phi\times {\rm sym}^{4}\phi)$. We then use the result \eqref{eq-HT-inequality} of Hall and Tenenbaum to get
	\begin{equation}\label{eq-ht-8}
	\sum_{\substack{n\leq X \\  (n, P)=1}}\lambda_{\phi}(n)^{8}\ll \frac{X}{\log X}\sum_{\substack{n\leq X \\  (n, P)=1}} \frac{\lambda_{\phi}(n)^{8}}{n}.
	\end{equation}
	By using the multiplicative property of $\lambda_{\phi}(n)$ and the estimate \eqref{eq-kgeq2}, the logarithmic average of $\lambda_{\phi}(n)^{8}$ can be controlled by
	\begin{equation}\label{eq-8sieve-prod}
	\begin{aligned}
	\sum_{\substack{n\leq X \\  (n, P)=1}} \frac{\lambda_{\phi}(n)^{8}}{n} & \ll \prod_{\substack{p \leq X \\
			p \nmid P}}\left(1+\frac{\lambda_{\phi}(p)^{8}}{p}+\sum_{k \geq 2} \frac{\lambda_{\phi}(p^{k})^8}{p^{k}}\right) \\
	& \ll \exp \bigg(\sum_{\substack{p \leq X \\
			p \nmid P}}\Big(\frac{\lambda_{\phi}(p)^{8}}{p}+\sum_{k \geq 2} \frac{\lambda_{\phi}(p^{k})^8}{p^{k}}\Big)\bigg) \\
	& \ll \exp \bigg(\sum_{\substack{p \leq X \\
			p \nmid P}} \frac{\lambda_{\phi}(p)^8}{p}\bigg) .
	\end{aligned}
	\end{equation}
	By \cite[Propostion 2.3]{RS}  of Rudnick and Sarnak, and  summation by parts, we obtain
	\begin{equation}\label{eq-merten-auto}
	\sum_{p\leq X}\frac{\lambda_{{\rm sym}^{j}\phi\times {\rm sym}^{j}\phi }(p)}{p}=\log\log X+O(1),
	\end{equation}
	where $1\leq j\leq 4$, and we use the fact that Hypothesis H of Rudnick and Sarnak holds for $\rm{sym}^{j}\phi$ with $1\leq j\leq 4$. By \eqref{eq-identity-8} and \eqref{eq-merten-auto}, we then derive
	$$
	\begin{aligned}
	\sum_{\substack{p \leq  X\\
			p \nmid P}} \frac{\lambda_{\phi}(p)^8}{p}&=\sum_{p \leq X} \frac{\lambda_{\phi}(p)^8}{p}-\sum_{p \leq \exp((\log X)^{1-\varepsilon/2})} \frac{\lambda_{\phi}(p)^8}{p}+\sum_{p \leq \exp((\log X)^{\varepsilon/2})} \frac{\lambda_{\phi}(p)^8}{p}\\
	&=14\varepsilon \log \log X+O(1).
	\end{aligned}
	$$
	Combining this, \eqref{eq-8sieve-prod} with \eqref{eq-ht-8}, we have
	$$
	\sum_{\substack{n\leq X \\  (n, P)=1}}\lambda_{\phi}(n)^{8}\ll \frac{X}{(\log X)^{1-14\varepsilon}}.
	$$
	By the H\"older inequality and Mertens' theorem, we also have
	$$
	\sum_{\substack{n\leq X \\  (n, P)=1}}\lambda_{\phi}(n)^{6}\ll  \Big(\sum_{\substack{n\leq X \\  (n, P)=1}}\lambda_{\phi}(n)^{8} \Big)^{\frac{3}{4}} \; \Big(\sum_{\substack{n\leq X \\  (n, P)=1}}1 \Big)^{\frac{1}{4}} \ll\frac{X}{(\log X)^{1-11\varepsilon}}.
	$$
	These complete the verification of Hypothese (ii) for $\lambda_{\phi}(n)^l$ with $l=3,4$.
	
	Now we can employ Theorem \ref{thm-main}, and then get
	\begin{equation*}
	\begin{aligned}
	M_{\phi,l}(X)=&h(0) \sum_{n \leq X}\lambda_{\phi}(n)^l  g_{0}(n)\log n+2\big(\gamma h(0)+h^\prime(0)\big) \sum_{n \leq X}\lambda_{\phi}(n)^l  g_{0}(n)\\
	&+2h(0) \sum_{n \leq X}\lambda_{\phi}(n)^l  g_{0}^{\prime}(n)+O\big(X(\log X)^{\frac{1}{2}+\varepsilon}\big),
	\end{aligned}
	\end{equation*}
	where $l=3,4$. The main term can be computed as in the proof of Theorem \ref{thm-sl2power-holomorphic}, so we omit the details here.

%
%

\end{proof}

\vskip 5mm

\section{Some lemmas}

\subsection{Classical lemmas}
In this section, we will state some results from the literature we need. The first lemma provides an upper bound on short sums of the divisor function in arithmetic progressions.

\begin{lem}\label{lem-BT}
	Let $b$ be arbitrary positive integer and let $q,l$ be integers. Suppose that $Y\leq X, q\leq YX^{-\varepsilon}$.
	We have
	$$
	\sum_{\substack{X-Y<n \leq X\\ n \equiv l \moq}} \tau(n)^{b} \ll \frac{\tau\big((l,q)\big)^{b}X}{q}\left(\frac{\varphi(q)}{q} \log X\right)^{2^{b}-1}
	$$
	uniformly in $l,q$, where the implied constant depends only on $\varepsilon$.
\end{lem}
\begin{proof}
	This lemma immediately follows from Shiu's classical estimate \cite[Theorem 1]{Shiu} for multiplicative functions and Mertens' theorem.
\end{proof}

The next lemma gives an asymptotic formula for divisor function in arithmetic progression with an explicit main term.

\begin{lem}\label{lem-divisor-ap}
	Denote $D(x ; q, l)$ by
	$$
	D(X ; q, l)=\sum_{\substack{n \leq X \\ n \equiv l\moq}} \tau(n).
	$$
	Then we have
	$$
	D(X; q, l)= X q^{-1} \sum_{h \mid q} c_{h}(l) h^{-1}(\log X+2 \gamma-1-2 \log h) +O\Big(X^{\frac{1}{3}+\varepsilon}\Big)
	$$
	provided $q\ll X^{\frac{2}{3}-\varepsilon}$, where $c_{h}(l)$ is the Ramanujan sum given by
	\[
	c_{q}(l)=\sideset{}{^*}\sum_{h\moq}e \Big(\frac{hl}{q}\Big)=\sum_{d \mid(q, l)} d\, \mu\Big(\frac{q}{d}\Big).
	\]
\end{lem}

\begin{proof}
	This is a well known result when $(l,q)=1$. Pongsriiam and Vaughan \cite[Theorem 1.1]{PonVaughan-2015} generalized it to the case of $(l,q)>1$.
\end{proof}

We cite below a result of Motohashi \cite[Lemma 8]{Motohashi-1980} for sums of Ramanujan sums.

\begin{lem}\label{lem-sum-mod}
	Let $m$ be an integer whose prime factors are all larger than $M$, and let $a$ be an arbitrary non-zero integer. We put
	$$
	Y(a)=\sum_{h=1}^{\infty}\left|c_{h}(a)\right| h^{-2}, \\\quad\quad
	Y^{(j)}(X ; a, m)=\sum_{\substack{h \leq X \\ (h, m)=1}} c_{h}(a) h^{-2}(\log h)^{j}.
	$$
	Then we have
	$$
	\begin{aligned}
	Y(a) &=O\left(\sigma_{-1}^{(0)}(a)\right) \\
	Y^{(j)}(x ; a, m) &=\sum_{r=0}^{j} \sigma_{-1}^{(r)}(a) \alpha_{r}(j)+O\left\{\left(\frac{d(m)}{M}+\frac{d(a)}{ X}\right)(\log aX  M)^{j+1}\right\},
	\end{aligned}
	$$
	where the constant $\alpha_{r}(j)$ depends only on $r, j$, and the function $\sigma_{-1}^{(r)}(a)$ is defined by
	$
	\sigma_{-1}^{(r)}(a)=\sum_{d |a}(\log d)^{r} d^{-1}.
	$
\end{lem}

Our next lemma is a truncated version of the Poisson summation formula
in arithmetic progressions.

\begin{lem} \label{lem-poisson} Let $\psi: \RR \longrightarrow [0,1]$ be smooth and compactly supported in $[1/2, 5/2]$ such that $\psi(t) = 1$ for $1 \leq t \leq 2$.  Then one has the equality
	\begin{equation}\label{eq-poi}
	\sum_{m\equiv a \moq} \psi \Bigl( \frac{m}{M}\Bigr) =\hat{\psi}(0)  \frac {M}{q}
	+ \frac{M}{q} \sum_{0 < \vert h \vert \leq H}e \bigl( \frac{ ah}{q} \bigr)\hat \psi \Bigl( \frac{h}{q/M}\Bigr) +O(M^{-1})
	\end{equation}
	for any $H\geq  (q/M)\log ^4 2M$.
	Moreover, one has the equality
	\begin{equation}\label{eq-poi-p}
	\sum_{(m,q)=1} \psi \Big( \frac{m}{M}\Big) =\frac{\varphi (q)}{q}\hat{\psi}(0)  M + O \big(\tau (q) (\log 2M)^4 \big),
	\end{equation}
	where $\hat{\psi}$ is the Fourier transform of $\psi$ defined by
	$$
	\hat{\psi}(\xi) = \int_{-\infty}^\infty \psi(t) e( -\xi t) \, \d t.
	$$
\end{lem}

\begin{proof}
	See \cite[Lemma 2.1]{FR-2022}.
\end{proof}

The next lemma, due to Bettin and Chandee \cite{BC-2018}, provides a non-trivial bound for trilinear form with Kloosterman fractions.

\begin{lem}\label{lem-trilinear}
	Let $\boldsymbol \alpha = \{\alpha_{m}\}$, $\boldsymbol \beta = \{\beta_{n}\}$, and $\boldsymbol \nu =\{\nu_{a}\}$ be three sequences of complex numbers.  For any non-zero integer $\vartheta$ and any $\varepsilon > 0$, we have,
	\[
	\begin{aligned}
	\Bigl\vert \, \sum_{a\sim A} \sum_{m\sim M} \sum_{n \sim N} \alpha (m) \beta (n) \nu (a) e \Bigl(\vartheta \frac{a \overline m}{n} \Bigr)\,   \Bigr\vert \ll  \Big(\sum_{a\sim A} |\nu (a) |^2 \Big)^{\frac{1}{2}} \Big(\sum_{m\sim M} |\alpha (m) |^2 \Big)^{\frac{1}{2}} \Big(\sum_{n\sim N} |\beta (n) |^2 \Big)^{\frac{1}{2}}
	\\
	\times  \Bigl( 1+ \frac{\vert \vartheta\vert A}{MN}\Bigr)^\frac{1}{2}  \Bigl( (AMN)^{\frac{7}{20} + \varepsilon} \, (M+N) ^\frac{1}{4} +(AMN)^{\frac{3}{8} +\varepsilon} (AN+AM) ^\frac{1}{8}
	\Bigr).
	\end{aligned}
	\]
\end{lem}
\begin{proof}
	See \cite[Theorem 1]{BC-2018}.
\end{proof}

The next lemma is related to the distribution of primes in arithmetic progressions. The key point here is that the moduli can be chosen larger than $\sqrt{X}$ and in particular that $a$ can be chosen even a little larger than $X$.

\begin{lem}\label{lem-bvprime}
	Let $X\geq2$ and $a \in \mathbb{Z} \backslash\{0\}$. There exists some positive constant $\delta$ such that
	\[
	\sum_{\substack{q \leq X^{1 / 2+\delta} \\(q, a)=1}}\bigg(\sum_{\substack{n \leq X \\ n \equiv a \moq}} \Lambda(n)-\frac{1}{\varphi(q)} \sum_{\substack{n \leq X}} \Lambda(n)\bigg) \ll_{A} \frac{X}{(\log X)^{A}}
	\]
	holds uniformly in $|a| \leq X^{1+\delta}$, where $A>0$ is an arbitrary constant.
\end{lem}
\begin{proof}
	See \cite[Theorem 2.1]{FR-2022}.
\end{proof}

The next lemma is a simple estimate and is useful to deal with the main terms in our theorems.

\begin{lem}\label{lem-phisum}
	For $X\geq2$ and $a \in \mathbb{Z} \backslash\{0\}$, we have
	$$
	\sum_{\substack{m \leq x \\
			(m, a)=1}} \frac{1}{\varphi(m)}=c(0, a)(\log x+\gamma)+c^{\prime}(0, a)+O\left(\tau(a) x^{-1} \log x\right)
	$$
	where $c(s, a)$ is given by
	$c(s, a)=\prod_{p \mid a}\left(1-\frac{1}{p^{s+1}}\right) \prod_{p \nmid a}\left(1+\frac{1}{(p-1) p^{s+1}}\right)$.
\end{lem}
\begin{proof}
	See \cite[Lemma 5.1]{ABL-2021}.
\end{proof}

\subsection{Sieve condition}

We denote by $P(Y,Z)$ the product of these primes $p$ which belong to the interval $[Y,Z)$,
where $2\leq Y<Z$.
In this section, we shall seek an upper estimate for the sifted sum
\begin{equation}\label{sifted-pi}
\sum_{\substack{n\leq X \\  \left(n-a, P(Y,Z)\right)=1}}\tau(n)^2
\end{equation}
with $0<|a|\leq X$ and $\big(a,P(Y,Z)\big)=1$. It is clear that $\tau(n)^2\leq \tau_4(n)$. Thus, it suffices to estimate the corresponding sum of \eqref{sifted-pi}, in which $\tau_4(n)$ is instead of $\tau(n)^2$. For this purpose, we require the distribution for $\tau_4(n)$ in arithmetic progressions.
Recall a result given in \cite{Smith-1982}, which states that
\begin{equation}\label{eq-tau4-ap}
\sum_{\substack{n\leq X\\  n\equiv a \moq}} \tau_4(n)=\frac{1}{\varphi(q)} X P_{4}(\log X)+O\Big(X^{\frac{3}{5}+\varepsilon}\Big)
\end{equation}
holds for any $(a,q)=1$ and $q\leq X^{\frac{2}{5}-\varepsilon}$, where $P_{4}(\log X)$ is a polynomial in $\log X$ of degree $3$ with real coefficients that depend on $q$. More precisely, the polynomial $P_{4}(\log X)$ is defined by
\[
P_{4}(\log X)=\sum_{j=0}^{3}\frac{1}{j!}B_{3-j}(\log X)\psi_q^{(j)}(1)
\]
with $\psi_q(s)=\big(\sum\limits_{k|q}k^{-s}\mu(k)\big)^4.$

Assume that $2\leq Y<Z\leq X^{\frac{1}{10}}$. We first remove the sieve condition $\left(n-a, P(Y,Z)\right)=1$ by the Legendre formula
\[
\sum_{\substack{n\leq X \\  \left(n-a, P(Y,Z)\right)=1}}\tau_4(n)=\sum_{q|P(Y,Z)}\mu(q)\sum_{\substack{n\leq X\\  n\equiv a \moq}} \tau_4(n).
\]
We keep the terms for $q|P(Y,Z)$ with $q\leq X^{\frac{1}{3}}$ and estimate the remaining ones. Applying Rankin's trick and the trivial bound $\tau_4(n)\ll n^{\varepsilon}$, we obtain
\begin{equation*}\label{mod-largeq}
\begin{aligned}
\sum_{\substack{q|P(Y,Z)\\ q>X^{1/3}}}|\mu(q)|\sum_{\substack{n\leq X\\  n\equiv a \moq}} \tau_4(n)\ll & X^{1+\varepsilon}\sum_{\substack{q|P(Y,Z)\\ q>X^{1/3}}}\frac{|\mu(q)|}{q}\ll X^{1-\varepsilon}\sum_{\substack{q|P(Y,Z)\\ q>X^{1/3}}}\frac{|\mu(q)|}{q^{1-6\varepsilon}}\\
\ll & X^{1-\varepsilon}\prod_{Y\leq p\leq Z}\Big(1+\frac{1}{p^{1-6\varepsilon}}\Big)\ll X^{1-\frac{2\varepsilon}{5}}.
\end{aligned}
\end{equation*}
To estimate the terms with $q\leq X^{1/3}$, we use the asymptotic formula \eqref{eq-tau4-ap} and get
\begin{equation}\label{mod-smallq}
\begin{aligned}
\sum_{\substack{q|P(Y,Z)\\ q\leq X^{1/3}}}\mu(q)\sum_{\substack{n\leq X\\  n\equiv a \moq}} \tau_4(n)=X\sum_{j=0}^{3}\frac{1}{j!}B_{3-j}(\log X) \sum_{\substack{q|P(Y,Z)\\ q\leq X^{1/3}}}\frac{\mu(q)\psi_q^{(j)}(1)}{\varphi(q)}+O\Big(X^{1-\varepsilon}\Big)
\end{aligned}
\end{equation}
After differentiating $\psi_q(s)$, we find that $\psi_q^{(j)}(1)/\varphi(q)\ll (\log\log q)^j/q$. Then applying Rankin's trick again, the truncation $q\leq X^{1/3}$ in the right-hand side of \eqref{mod-smallq} can be removed up to error term $X^{1-\frac{2\varepsilon}{5}}$. In addition, an elementary calculation arrives at
\begin{equation*}
\begin{aligned}
\sum_{q|P(Y,Z)}\frac{\mu(q)\psi_q^{(j)}(1)}{\varphi(q)}=&\bigg(\sum_{q|P(Y,Z)}\frac{\mu(q)\psi_q(s)}{\varphi(q)}\bigg)^{(j)}\bigg|_{s=1}=\bigg(\prod_{Y\leq p\leq Z}\Big(1-\frac{\psi_p(s)}{p-1}\Big)\bigg)^{(j)}\bigg|_{s=1}\\
\ll &\frac{\log Y}{\log Z}.
\end{aligned}
\end{equation*}

Combining all the above estimates, we can derive the following lemma.

\begin{lem}\label{lem-sieve-divisor}
	Suppose that $0< Y < Z \leq X^{\frac{1}{10}}$. Let $P(Y,Z)$ denote the product of primes $p$ which belong to the interval $[Y,Z)$. Then we have
	$$
	\sum_{\substack{n\leq X \\  \left(n-a, P(Y,Z)\right)=1}}\tau(n)^2  \ll X(\log X)^3\frac{\log Y}{\log Z}
	$$
	for any non-zero $|a|\leq X$ and $\big(a,P(Y,Z)\big)=1$.
\end{lem}
\vskip 5mm

\section{Reduction of Theorem \ref{thm-main}}

\smallskip

For convenience, we introduce a new notation
\[
S_{f}(X,\ell)=\sum_{n\leq X} f(n) \tau(n+\ell),
\]
where $\ell$ is a non-zero integer. First, we make a heuristic observation to illustrate the idea of Motohashi \cite{Motohashi-1980}. By Dirichlet's hyperbola method, we may write
\begin{equation}\label{observ-1}
S_{f}(X,-1)\sim 2\sum_{d\leq \sqrt{X}}\sum_{\substack{n\leq X \\  n\equiv 1\modd}}f(n).
\end{equation}
From now on we assume that $h$ is arbitrary prime number in the interval $[X^{2/3},  X(\log X)^{-A}]$, where $A$ is any sufficiently large positive constant. Let us consider $S_{f}(X,h)$. We may also write
\begin{equation}\label{observ-2}
S_{f}(X,h)\sim 2\sum_{d\leq \sqrt{X}}\sum_{\substack{n\leq X \\  n\equiv -h\modd}}f(n).
\end{equation}
The inner sums of \eqref{observ-1} and \eqref{observ-2} can be put into the character sum
\begin{equation}\label{observ-3}
\varphi(d)^{-1} \sum_{n \leq X \atop(d, n)=1}f(n)+\varphi(d)^{-1} \sum_{\substack{\chi\modd \\ \chi \neq \chi_{0}}} \bar{\chi}(-\ell) \sum_{n \leq X} f(n)\chi(n),
\end{equation}
where $\ell=-1$ or $\ell=h$ as above. Note that $(\ell, d)=1$ by the choice of $h$. Thus, they may well be approximately equal to the first term of \eqref{observ-3}. In other words, one may expect that $S_{f}(X,-1)$ differs little from $S_{f}(X,h)$ for each prime $h \in[X^{2/3},  X(\log X)^{-A}]$. If so, then $S_{f}(X,-1)$ will be approximately equal to
$$
\pi\Big(\frac{X}{(\log X)^{A}}\Big)^{-1} \sum_{X^{2/3}\leq h\leq   X(\log X)^{-A}} S_{f}(X,h),
$$
where $\pi\big(X(\log X)^{-A}\big)$ denotes the number of primes less than $X(\log X)^{-A}$. This sum obviously belongs to the category of ternary problems, which can be addressed by the Hardy--Littlewood circle method.

In fact, after overcoming certain technical difficulties, it will turn out that $|S_{f}(X,-1)-S_{f}(X,h)|$ is relatively small, i.e. less than $X(\log X)^{1/2+\varepsilon}$ uniformly for any prime $h \in[X^{2/3},  X(\log X)^{-A}]$.

Now we begin to provide the rigorous procedure. Let $\varepsilon$ be any sufficiently small positive constant. Let $\Delta\in [(\log X)^{-A}, 2 (\log X)^{-A}]$ be a real number such that the number
$$
L:= \frac{(\log X)^{1-\frac{\varepsilon}{2}} - (\log X)^{\frac{\varepsilon}{2}}}{\log(1 + \Delta)}
$$
is an integer. It is clear that $L \ll (\log X)^{A + 1}$. We then can partition the interval $[\exp((\log X)^{\varepsilon/2}), \\ \exp((\log X)^{1-\varepsilon/2})]$ into at most $L$  intervals
\[
\mathcal{I}_{\nu} := [\exp((\log x)^{\varepsilon/2}) (1 + \Delta)^{\nu}, \exp((\log x)^{\varepsilon/2}) (1 + \Delta)^{\ell + 1}) ) := [H_{\nu}, H_{\nu + 1})
\]
with $0\leq \nu\leq L$. Moreover, we set
\begin{equation*}
\ P_\nu = \prod_{H_0\leq p <H_{\nu + 1}} p, \quad \P_\nu = \{p\in \I_\nu\}, \quad \M_\nu = \Big\{m\in \big[1,\frac{X}{H_{\nu + 1}}\big]: (m, P_\nu)=1 \Big\},
\end{equation*}

\begin{equation*}
\P_\nu\M_\nu = \{pm: p\in\P_\nu, m\in\M_\nu\},
\end{equation*}

\begin{equation}\label{eq-partition-IJ}
\I= \bigcup_{0\leq \nu\leq L} \P_\nu\M_\nu \ \ \text{and} \ \ \J= [1,X] \setminus \I.
\end{equation}
The intervals above are always meant as subsets of $\NN$. Notice that each $n\in \P_\nu\M_\nu$ can be written in a unique way as $n=pm$ with $p\in\P_\nu$ and $m\in\M_\nu$. Thus, we get $|\P_\nu\M_\nu| = |\P_\nu||\M_\nu|$ and $\P_\nu\M_\nu\subset [1,X]$. Moreover, the sets $\P_\nu\M_\nu$ are pairwise disjoint for $0\leq \nu\leq L$.

We are first concerned with the behavior of $S_{f}(X,\ell)$. By the decomposition above of $[1,X]$, the sum $S_{f}(X,\ell)$ is divided into two parts as follows
\begin{equation*}\label{2-4}
\begin{aligned}
S_{f}(X,\ell)=& \sum_{n\in\I} f(n) \tau(n+\ell) + \sum_{n\in\J} f(n) \tau(n+\ell) \\
:=& S_{f}^{\I}(X,\ell)+ S_{f}^{\J}(X,\ell).
\end{aligned}
\end{equation*}
The aim is to turn out that $|S_{f}(X,-1)-S_{f}(X,h)|$ is relatively small for any prime $h\in [X^{2/3},  X(\log X)^{-A}]$. So we shall estimate the difference of these two sums over ${\I}$ and $\J$, respectively.
\vskip 5mm

\section{Evaluation of $S_{f}^{\I}(X,-1)-S_{f}^{\I}(X,h)$}

Since $\P_\nu\M_\nu$ are pairwise disjoint for $0\leq \nu\leq L$, we obtain that
\begin{equation}\label{2-5}
S_{f}^{\I}(X,-1)-S_{f}^{\I}(X,h)\ll  \sum_{0\leq \nu\leq L}\Big| \sum_{pm\in\P_\nu\M_\nu} f(pm)\big(\tau(pm-1)-\tau(pm+h)\big)\Big|.
\end{equation}
For $pm\in\P_\nu\M_\nu$, we have $p\in \P_\nu$ and $m\in\M_\nu$. The $f(pm)$ in \eqref{2-5} can be factored as $f(p)f(m)$ by its multiplicativity. Thus, we get
\begin{equation*}
S_{f}^{\I}(X,-1)-S_{f}^{\I}(X,h) \ll \sum_{0\leq \nu\leq L}\sum_{m\in\M_\nu} |f(m)|\, \Big|\sum_{p\in\P_\nu} f(p)\big(\tau(pm-1)-\tau(pm+h)\big)\Big|.
\end{equation*}
Using the 
Cauchy--Schwarz inequality and exchanging the order of summations, the inner sum can be performed as follows  
\begin{equation}\label{2-6}
\begin{aligned}
&\sum_{m\in\M_\nu} |f(m)| \Big|\sum_{p\in\P_\nu} f(p)\big(\tau(pm-1)-\tau(pm+h)\big)\Big| \\
\leq &\Big(\sum_{m\in\M_\nu} |f(m)|^2\Big)^{\frac{1}{2}}  \Big(\sum_{m\in\M_\nu} \Big|\sum_{p\in\P_\nu} f(p)\big(\tau(pm-1)-\tau(pm+h)\big) \Big|^2 \Big)^{\frac{1}{2}} \\
\leq&  \Big(\sum_{m\in\M_\nu} |f(m)|^2\Big)^{\frac{1}{2}}  \Big(\sum_{m\leq X/H_{\nu + 1}} \Big|\sum_{p\in\P_\nu} f(p)\big(\tau(pm-1)-\tau(pm+h)\big) \Big|^2 \Big)^{\frac{1}{2}}\\
= & \Big(\sum_{m\in\M_\nu} |f(m)|^2\Big)^{\frac{1}{2}} \Big(\sum_{p_1,p_2\in\P_\nu}f(p_1)\overline{f(p_2)} \, \big( V_1-V_2-V_3+V_4 \big)\Big)^{\frac{1}{2}},
\end{aligned}
\end{equation}
where the four terms $V_j \,(j=1,2,3,4)$ are given by
\begin{equation*}
\begin{aligned}
&V_1=\sum_{m\leq X/H_{\nu + 1}}\tau(p_1m-1)\tau(p_2m-1), \quad  V_2=\sum_{m\leq X/H_{\nu + 1}}\tau(p_1m-1)\tau(p_2m+h),\\
&V_3=\sum_{m\leq X/H_{\nu + 1}}\tau(p_1m+h)\tau(p_2m-1), \quad V_4=\sum_{m\leq X/H_{\nu + 1}}\tau(p_1m+h)\tau(p_2m+h).
\end{aligned}
\end{equation*}

The diagonal contribution in the last line of \eqref{2-6}, that is $p_1=p_2$ for each $\nu$ yields at most
\begin{equation*}\label{2-6-1}
\begin{split}
&\Big(\sum_{m\in\M_\nu} |f(m)|^2\Big)^{\frac{1}{2}} \Big(\sum_{p\in\P_\nu}|f(p)|^2 \sum_{\substack{n\leq X \\  n\equiv 1, -h\mop}}|\tau(n)|^2\Big)^{\frac{1}{2}}\\
\ll  &\frac{X^{\frac{1}{2}} (\log X)^{2}}{H_{\nu}^\frac{1}{2}} \Big(\sum_{p\in\P_\nu,m\in\M_\nu} |f(pm)|^2\Big)^{\frac{1}{2}} ,
\end{split}
\end{equation*}
by Lemma \ref{lem-BT} and the multiplicative property of $f(n)$. Hence, by summing over $\nu$ and the Cauchy--Schwarz inequality, the diagonal contribution to $S_{f}^{\I}(X,-1)-S_{f}^{\I}(X,h)$ is less than
\begin{equation}\label{dia-contr}
\begin{split}
&X^{\frac{1}{2}}(\log X)^{2}\Big(\sum_{0\leq \nu\leq L}H_{\nu}^{-1}\Big)^{\frac{1}{2}}\Big(\sum_{0\leq \nu\leq L}\sum_{p\in\P_\nu,m\in\M_\nu} |f(pm)|^2\Big)^{\frac{1}{2}}\\
\ll &X^{\frac{1}{2}}\exp\big(-(\log x)^{\frac{\varepsilon}{2}}\big)\Big(\sum_{n\leq X} |f(n)|^2\Big)^{\frac{1}{2}}\\
\ll &X\exp\big(-(\log x)^{\frac{\varepsilon}{3}}\big).
\end{split}
\end{equation}

Now we turn our attention to the off-diagonal terms. 
For $p_1\neq p_2$, the sums $V_i$ for $1\leq i\leq 4$ are shifted convolution sums associated to the divisor function. 
Therefore, we have to estimate
$$
V(\varpi_{1}, \varpi_{2})=\sum_{m \leq X/H_{\nu + 1}} \tau(p_1 m+\varpi_{1}) \tau(p_2 m+\varpi_{2}),
$$
where $p_1,p_2\in\P_\nu$ with $p_1\neq p_2$, and $\left(\varpi_{1}, \varpi_{2}\right)$ is one of the pairs
$
(-1,-1),(-1, h), (h, -1),(h, h).
$
Putting $X_{1}=p_1X/H_{\nu + 1}-\varpi_{1}$ and using the notation of Lemma \ref{lem-divisor-ap}, we have
\begin{equation}\label{V-eq}
\begin{aligned}
V(\varpi_{1}, \varpi_{2})=& \sum_{\substack{d \leq \sqrt{X_{1}}\\ (d,p_1)=1}}\left(2 D\big(p_2X/H_{\nu + 1}-\varpi_{2} ; d p_2, \varpi_{2}-\varpi_{1} \bar{p_1} p_2\big)\right.\\
&\left.-D\big(d \sqrt{X_{1}} p_2/p_1+\left(\varpi_{2} p_1-\varpi_{1} p_2\right) / p_1; d p_2, \varpi_{2}-\varpi_{1} \bar{p_1} p_2\big)\right)\\
:=& 2 V^{(1)}(\varpi_{1}, \varpi_{2})-V^{(2)}(\varpi_{1}, \varpi_{2}),
\end{aligned}
\end{equation}
where $p \bar{p} \equiv 1(\bmod d)$. By Lemma \ref{lem-BT}, one has
$$
D\left(p_2X/H_{\nu + 1}-\varpi_{2} ; d p_2, \varpi_{2}-\varpi_{1} \bar{p_1} p_2\right) \\
=D\big(X; dp_2, \varpi_{2}-\varpi_{1} \bar{p_1} p_2\big)+O\Big(\frac{\tau(d)X}{dp_2 (\log X)^{A-2} }\Big).
$$
Moreover, the terms corresponding to $d$ such that $p_2 | d$ in $V^{(1)}(\varpi_{1}, \varpi_{2})$ contribute by the amount $O\left(X(\log X)^{4}/H_{\nu}^{2}\right)$, in which  Lemma \ref{lem-BT} is used again.  Noticing that $X_{1}=X\big(1+O\big((\log X)^{-A}\big)\big)$, we further have
\[
V^{(1)}\left(N ; v_{1}, v_{2}\right)=\sum_{\substack{d \leq \sqrt{X} \\ (d, p_1p_2)=1}} D\big(X; d p_2, \varpi_{2}-\varpi_{1} \bar{p_1} p_2\big)+O\Big(\frac{X}{H_{\nu} (\log X)^{A-4} }\Big).
\]
Before appealing to Lemma \ref{lem-divisor-ap}, on the condition $(d, p_1p_2)=1$, we have
\[
c_{r}\left(\varpi_{2}-\varpi_{1} \bar{p_1} p_2\right)= \begin{cases}c_{r}\left(\varpi_{2} p_1-\varpi_{1} p_2\right) & \text { if } r \mid d, \\ -c_{r_{1}}\left(\varpi_{2} p_1-\varpi_{1} p_2\right) & \text { if } r=p_2 r_{1}, r_{1} \mid d .\end{cases}
\]
Hence, we get from Lemma \ref{lem-divisor-ap} that
$$
\begin{aligned}
D\big(X; d p_2, \varpi_{2}-\varpi_{1} \bar{p_1} p_2\big)
=& \frac{X}{dp_2}\sum_{r|d} \frac{c_{r}\left(\varpi_{2} p_1-\varpi_{1} p_2\right)}{r} (\log X+2 \gamma-1-2 \log r) \\
&+O\Big(\frac{X\log X}{dH_{\nu}^{2}} \sum_{r |d} \frac{|c_{r}\left(\varpi_{2} p_1-\varpi_{1} p_2\right)|}{r}\Big)+O\Big(X^{\frac{1}{3}+\varepsilon}\Big).
\end{aligned}
$$
Then, after some elementary computations, we see that, using the notation of Lemma \ref{lem-sum-mod},
\begin{equation}\label{V1-eq}
\begin{aligned}
V^{(1)}(\varpi_{1}, \varpi_{2})
=&  \frac{X}{2p_2}\big(\log X+2 \gamma-1\big)\big(\log X+2 \gamma\big) Y^{(0)}\big(\sqrt{X} ; \varpi_{2} p_1-\varpi_{1} p_2,  p_1 p_2\big) \\
&- \frac{X}{p_2}\big(\log X+\log X+4 \gamma-1\big) Y^{(1)}\big(\sqrt{X} ; \varpi_{2} p_1-\varpi_{1} p_2,  p_1 p_2\big) \\
&+ \frac{2X}{p_2} Y^{(2)}\big(\sqrt{X} ; \varpi_{2} p_1-\varpi_{1} p_2,  p_1 p_2\big) \\
&+O\Big(\frac{X\log X}{H_{\nu}^{2}} Y(\varpi_{2} p_1-\varpi_{1} p_2 )\Big)+O\Big(\frac{X}{H_{\nu} (\log X)^{A-4} }\Big).
\end{aligned}
\end{equation}
In the same way, we just get
\begin{equation}\label{V2-eq}
\begin{aligned}
V^{(2)}(\varpi_{1}, \varpi_{2})
=&  \frac{X}{p_2}\big(\log X+2 \gamma-2\big)Y^{(0)}\big(\sqrt{X} ; \varpi_{2} p_1-\varpi_{1} p_2,  p_1 p_2\big) \\
&- \frac{3X}{p_2}Y^{(1)}\big(\sqrt{X} ; \varpi_{2} p_1-\varpi_{1} p_2,  p_1 p_2\big) \\
&+O\Big(\frac{X\log X}{H_{\nu}^{2}} Y(\varpi_{2} p_1-\varpi_{1} p_2 )\Big)+O\Big(\frac{X}{H_{\nu} (\log X)^{A-4} }\Big).
\end{aligned}
\end{equation}
Inserting \eqref{V1-eq} and \eqref{V2-eq} into \eqref{V-eq}, we have
\begin{equation*}
\begin{aligned}
V(\varpi_{1}, \varpi_{2})
=&  \frac{X}{p_2}\big((\log X+2 \gamma-1)(\log X+2 \gamma-1)+1\big) Y^{(0)}\big(\sqrt{X} ; \varpi_{2} p_1-\varpi_{1} p_2,  p_1 p_2\big) \\
&- \frac{2X}{p_2}\big(\log X+4 \gamma-\frac{5}{2}\big) \,Y^{(1)}\big(\sqrt{X} ; \varpi_{2} p_1-\varpi_{1} p_2,  p_1 p_2\big) \\
&+ \frac{4X}{p_2} Y^{(2)}\big(\sqrt{X} ; \varpi_{2} p_1-\varpi_{1} p_2,  p_1 p_2\big) \\
&+O\Big(\frac{X\log X}{H_{\nu}^{2}} Y(\varpi_{2} p_1-\varpi_{1} p_2 )\Big)+O\Big(\frac{X}{H_{\nu} (\log X)^{A-4} }\Big).
\end{aligned}
\end{equation*}
Applying Lemma \ref{lem-sum-mod}, we derive that there exist three polynomials $w_{j}(x)$of degree $j$ with numerical coefficients such that
\begin{equation*}
\begin{aligned}
V(\varpi_{1}, \varpi_{2})= \frac{X}{H_{\nu}} \sum_{j=0}^{2} \sigma_{-1}^{(2-j)}(\varpi_{2} p_1-\varpi_{1} p_2)w_{j}(\log X)
+O\Big(\frac{X}{H_{\nu} (\log X)^{A-4} } \sigma_{-1}^{(0)}(\varpi_{2} p_1-\varpi_{1} p_2)\Big).
\end{aligned}
\end{equation*}
We emphasize that the coefficients of $w_{j}$ do not depend on $p_1,p_2,\varpi_{1},\varpi_{2}$. Then we obtain
\begin{equation}\label{sum-pq}
\begin{aligned}
&\sum_{\substack{p_1,p_2\in\P_\nu\\ p_1\neq p_2}}f(p_1)\overline{f(p_2)} \, \big( V_1-V_2-V_3+V_4 \big)\\
\ll &  \frac{X(\log X)^{2}}{H_{\nu}} \sum_{j=0}^{2}\Big| \sum_{\substack{p_1,p_2\in\P_\nu\\ p_1\neq p_2}}f(p_1)\overline{f(p_2)}   \big(\sigma_{-1}^{(j)}(p_2-p_1)-\sigma_{-1}^{(j)}(hp_2-p_1)-\sigma_{-1}^{(j)}(p_2-hp_1)\\
&+\sigma_{-1}^{(j)}(h(p_2-p_1))\big) \Big|+\frac{X}{H_{\nu}(\log X)^{A-4} }\max_{\substack{\varpi_{1}=-1,h\\\varpi_{2}=-1,h}}\sum_{p_1\in\P_\nu}|f(p)|^2\sum_{p_2\in\P_\nu} \sigma_{-1}^{(0)}(\varpi_{2} p_1-\varpi_{1} p_2).
\end{aligned}
\end{equation}
By the definition of $\sigma_{-1}^{(j)}$ and exchanging the order of summations, we have
\begin{equation*}
\begin{aligned}
&\sum_{\substack{p_1,p_2\in\P_\nu\\ p_1\neq p_2}}f(p_1)\overline{f(p_2)} \sigma_{-1}^{(j)}(p_2-p_1)\\
=& \sum_{p_1,p_2\in\P_\nu}f(p_1)\overline{f(p_2)} \sum_{\substack{d|(p_2-p_1)\\ d\leq (\log H_{\nu})^{B}}}\frac{(\log d)^j}{d}+O\Big(\frac{H_{\nu}}{(\log X)^A(\log H_{\nu})^{B-1}}\sum_{p\in\P_\nu}|f(p)|^2\Big)\\
=& \sum_{p_1\in\P_\nu}f(p_1)\sum_{d\leq (\log H_{\nu})^B }\frac{(\log d)^j}{d} \sum_{\substack{p_2\in\P_\nu \\ p_2\equiv p_1\modd}}\overline{f(p_2)}+O\Big(\frac{H_{\nu}}{(\log X)^A(\log H_{\nu})^{B-1}}\sum_{p\in\P_\nu}|f(p)|^2\Big).
\end{aligned}
\end{equation*}
The contributions of $\sigma_{-1}^{(j)}(hp_1+p_2), \sigma_{-1}^{(j)}(p_1+hp_2),\sigma_{-1}^{(j)}(h(p_1-p_2))$ can be computed in the same manner. Moreover,  the last term in \eqref{sum-pq} is less than $\frac{X}{H_{\nu}(\log X)^{2A-4} }\sum\limits_{p\in\P_\nu}|f(p)|^2$. Accordingly, we can obtain from Hypothesis (iii) that
\begin{equation*}\label{sum-pq-2}
\sum_{\substack{p_1,p_2\in\P_\nu\\ p_1\neq p_2}}f(p_1)\overline{f(p_2)} \, \big( V_1-V_2-V_3+V_4 \big)
\ll  \frac{X(\log X)^{2}}{(\log H_{\nu})^{B-1}}\sum_{p\in\P_\nu}\big(1+|f(p)|^2\big).
\end{equation*}
Here the  $B$ is an arbitrarily large constant. By inserting this into \eqref{2-6}, summing over $\nu$ and using the Cauchy--Schwarz inequality, the off-diagonal contribution to $S_{f}^{\I}(X,-1)-S_{f}^{\I}(X,h)$ is bounded by
\begin{equation}\label{offdia-contr}
\begin{split}
&X^{\frac{1}{2}}(\log X)\Big(\sum_{0\leq \nu\leq L}\frac{1}{(\log H_{\nu})^{B-1}}\Big)^{\frac{1}{2}}\Big(\sum_{0\leq \nu\leq L}\sum_{p\in\P_\nu,m\in\M_\nu} \big(|f(m)|^2+|f(pm)|^2\big)\Big)^{\frac{1}{2}}\\
&\ll \Big(\frac{X(\log X)^{A+1}}{(\log H)^{B-1}}\Big)^{\frac{1}{2}}\Big(\sum_{n\leq X} |f(n)|^2\Big)^{\frac{1}{2}}\\
& \ll \frac{X(\log X)^{\frac{A+c}{2}}}{(\log H)^{\frac{B-1}{2}}}.
\end{split}
\end{equation}
Therefore, by taking $B=(3A+c)/\varepsilon+1$ for any $A>0$, we obtain the following proposition from \eqref{dia-contr} and \eqref{offdia-contr}.

\begin{prop}\label{prop-sum-I}
	Let the notation be as above. Then we have
	\[
	S_{f}^{\I}(X,-1)-S_{f}^{\I}(X,h) \ll \frac{X}{(\log X)^A}
	\]
	uniformly for any prime $h \in [X^{2/3},  X(\log X)^{-A}]$ and any $A>0$.
\end{prop}

\vskip 5mm

\section{Evaluation of $S_{f}^{\J}(X,-1)-S_{f}^{\J}(X,h)$}
In order to estimate the contribution of corresponding sums over the set $\J$, we first define the following subsets of $[1,X]$:
\begin{equation*}
\begin{split}
\J_1^{(\nu)} &= \Big\{n\in[1,X]: n \ \text{has exactly one divisor in $\P_\nu$ and none in $\bigcup_{0\leq h < \nu}\P_h$}\Big\}, \\
\J_1 &= \bigcup_{0\leq \nu\leq L} \J_1^{(\nu)}, \\
\J_2 &= \Big\{n\in[1,X]: n \ \text{has at least one prime factor in $\bigcup_{0\leq \nu \leq L}\P_\nu$}\Big\},  \\
\J_3 &= \Big\{n\in[1,X]: n \ \text{has no prime factors in ${\bigcup_{0\leq \nu \leq L}\P_\nu}$}\Big\}.
\end{split}
\end{equation*}
By the definitions of these subsets of $[1,X]$, it is clear that $P_\nu\M_\nu \subset\J_1^{(\nu)}$. Thus we have $\I\subset\J_1$. Moreover, $\J_2\cup \J_3 =[1,X]$ and $\J_2\cap \J_3 =\emptyset$. Thus, we get
\begin{equation*}
\J \subset (\J_1\setminus \I) \cup (\J_2\setminus \J_1) \cup \J_3.
\end{equation*}
As a consequence, it follows from Hypothesis (i)  that
\begin{equation}\label{eq-J}
\begin{aligned}
S_{f}^{\J}&(X,\ell) \ll \sum_{n\in\J_1\setminus\I} |f(n)|\tau(n+\ell) + \sum_{n\in\J_2\setminus\J_1} |f(n)|\tau(n+\ell) + \Big|\sum_{n\in\J_3} f(n)\tau(n+\ell) \Big|\\
\ll& X(\log X)^{\frac{c-1}{2}}\bigg(\sum_{n\in\J_1\setminus\I} |\tau(n+\ell) |^2+\sum_{n\in\J_2\setminus\J_1} |\tau(n+\ell) |^2\bigg)^{\frac{1}{2}}+ \Big|\sum_{n\in\J_3} f(n)\tau(n+\ell) \Big|.
\end{aligned}
\end{equation}

For each $\nu \in [0, L]$, it is obvious that
\[
\J_1^{(\nu)}\setminus \\P_\nu\M_\nu \subset \P_\nu \Big(\frac{X}{H_{\nu+1}},\frac{X}{H_{\nu}}\Big].
\]
Hence, we have
\begin{equation*}
\sum_{n\in\J_1\setminus\I} |\tau(n+\ell) |^2 \ll \sum_{0\leq \nu\leq L} \sum_{p\in \P_\nu} \sum_{\substack{|n-X+\ell|\ll X(\log X)^{-A}\\  n\equiv -\ell \mop}}|\tau(n)|^2.
\end{equation*}
The innermost sum is actually related to the coefficients $|\tau(n)|^2$ in arithmetic progression over a short interval. By the inequality in Lemma \ref{lem-BT}, we deduce that
\[
\sum_{\substack{|n-X+\ell|\ll X(\log X)^{-A}\\  n\equiv -\ell \mop}}|\tau(n)|^2\ll \frac{X}{H_{\nu}(\log X)^{A-3}}.
\]
Further, we obtain
\begin{equation}\label{eq-J1}
\sum_{n\in\J_1\setminus\I} |\tau(n+\ell) |^2
\ll \frac{X}{(\log X)^{A-4}}.
\end{equation}
Moreover, on account of
\begin{equation*}
\J_2\setminus \J_1 \subset \bigcup_{0\leq \nu\leq L} \{n\in[1,X]: n \ \text{has at least two prime factors in $\P_\nu$}\},
\end{equation*}
we get
\begin{equation*}
\sum_{n\in\J_2\setminus\J_1} |\tau(n+\ell) |^2 \ll  \sum_{0\leq \nu\leq L} \sum_{p_1,p_2\in\P_\nu} \sum_{\substack{n \leq X+\ell\\  n\equiv -\ell \mopq}}|\tau(n)|^2.
\end{equation*}
Similar to the argument of \eqref{eq-J1}, we use Lemma \ref{lem-BT} again and then obtain
\begin{equation}\label{eq-J2}
\sum_{n\in\J_2\setminus\J_1} |\tau(n+\ell) |^2 \ll  X(\log X)^3  \sum_{0\leq \nu\leq L} \Big(\frac{|\P_\nu|}{H_{\nu}}\Big)^2
\ll \frac{X}{(\log X)^{A-4}}.
\end{equation}
Moreover, it follows from the Cauchy--Schwarz inequality that
\[
\sum_{n\in\J_3} f(n)\tau(n+\ell) \ll \bigg(\sum_{\substack{n\leq X \\ (n, P)=1}} |f(n)|^2 \bigg)^\frac{1}{2}\bigg(\sum_{\substack{n\leq X \\ (n-\ell, P)=1}} \tau(n)^2 \bigg)^\frac{1}{2},
\]
where $P=P\big(\exp((\log X)^{\varepsilon/2}), \, \exp((\log X)^{1-\varepsilon/2})\big)$.
Note that $\ell=-1$ or $\ell=h\in[X^{2/3},  X(\log X)^{-A}]$ is a prime, so $(\ell,P)=1$.  Directly applying Lemma \ref{lem-sieve-divisor} with $Y=\exp\big((\log X)^{\varepsilon/2}\big), Z=\exp\big((\log X)^{1-\varepsilon/2}\big), a=\ell$, and inserting Hypothesis (ii), we then get
\begin{equation}\label{eq-J3}
\sum_{n\in\J_3} f(n)\tau(n+\ell) \leq X(\log X)^{\frac{1}{2}+\varepsilon}.
\end{equation}

Combining \eqref{eq-J1}--\eqref{eq-J3} with \eqref{eq-J}, we finally conclude  the following proposition.
\begin{prop}\label{prop-sum-J}
	Let the notation be as above. For any prime $h\in[X^{2/3},  X(\log X)^{-A}]$, we have
	\[
	S_{f}^{\J}(X,-1)-S_{f}^{\J}(X,h) \ll X(\log X)^{\frac{1}{2}+\varepsilon}.
	\]
\end{prop}

\begin{rem}\label{rem-J}
	Assume the Ramanujan conjecture holds for $f(n)$, which means there exists some positive integer $k$ such that $|f(n)|\leq \tau_{k}(n)$ for all $n\in \NN.$ Opening the divisor function and using Dirichlet's hyperbola method, it reduces
	\[
	\sum_{n\in\J_3} f(n)\tau(n+\ell) =2\sum_{d\leq \sqrt{X-\ell}}\sum_{\substack{n\leq X,\, (n,P)=1\\ n\equiv \ell\modd}}f(n).
	\]
	Note that $\ell=1$ or $\ell=h\in[X^{2/3},  X(\log X)^{-A}]$ is a prime, so $(d,\ell)=1$. The Brun--Titchmarsh inequality \cite[Theorem 1]{Shiu} gives
	\[
	\sum_{\substack{n\leq X,\, (n,P)=1\\ n\equiv \ell\modd}}f(n)\ll \frac{X}{\varphi(d)\log X} \exp \Big(\sum_{\substack{p \leq X \\ p \nmid P}} \frac{|f(p)|}{p}\Big)\ll \frac{X}{\varphi(d)(\log X)^{1-k\varepsilon}}.
	\]
	A slight estimate of Titchmarsh [42, Equation (3.2)] states that
	$$
	\sum_{d \leq X} \frac{1}{\varphi(d)}=\frac{\zeta(2) \zeta(3)}{\zeta(6)} \log X+O(1).
	$$
	Hence, we have, under the Ramanujan conjecture for $f(n)$,
	\begin{equation*}\label{RC-J3}
	\sum_{n\in\J_3} f(n)\tau(n+\ell) \leq X(\log X)^{k\varepsilon}.
	\end{equation*}
	This yields that
	$
		S_{f}^{\J}(X,-1)-S_{f}^{\J}(X,h) \ll X(\log X)^{k\varepsilon}.
	$
\end{rem}

\vskip 5mm

\section{Proof of Theorem \ref{thm-main}}
By Proposition \ref{prop-sum-I} and Proposition \ref{prop-sum-J}, we obtain
\[
|S_{f}(X,-1)-S_{f}(X,h)|\ll X(\log X)^{\frac{1}{2}+\varepsilon}
\]
uniformly for any prime $h\in[X^{2/3},  X(\log X)^{-A}]$. Thus, we have
\begin{equation}\label{eq-1top}
S_{f}(X,-1)=\frac{(\log X)^{A}}{X}T_{f}(X)+O\big(X(\log X)^{\frac{1}{2}+\varepsilon}\big),
\end{equation}
where $T_{f}(X)$ is a ternary additive problem given by
\begin{equation*}\label{eq-defTf}
T_{f}(X)=\sum_{h\leq X(\log X)^{-A}}\sum_{n \leq X} \Lambda(h)f(n)\tau(n+h).
\end{equation*}
Hence, the additive divisor problem that we are considering is reduced to the study of a ternary additive problem $T_{f}(X)$. The Hardy--Littlewood circle method can be applied. Here we provide another direct argument. Notice that the interval of summation over $h$ is large enough compared to that over $n$. So we may exchange the order of summations in $T_{f}(X)$. Then it reduces to investigate the original divisor problem of Titchmarsh
\[
T(Y,n)=\sum_{h\leq Y}  \Lambda(h)\tau(h+n),
\]
where the shift parameter $n$ is in the interval $0<n\ll Y(\log Y)^A$.

By the definition of divisor function, we have
\begin{equation}\label{eq-def-tau}
\tau(m)=2 \sum_{\substack{q \mid m \\
		q<\sqrt{m}}} 1+\delta_{\square}(m), \quad \delta_{\square}(m):=\left\{\begin{array}{ll}
1 & \text { if } m \text { is a perfect square,} \\
0 & \text { otherwise. }
\end{array}\right.
\end{equation}
Therefore, we deduce that
\begin{equation}\label{eq-TYn}
\begin{aligned}
T(Y,n)=&2\Big( \sum_{\substack{q \leq \sqrt{Y+n} \\ (q,n)=1}}\sum_{\substack{h\leq Y\\ h\equiv -n\moq}}\Lambda(h) -\sum_{\substack{\sqrt{n}<q \leq \sqrt{Y+n}\\ (q,n)=1}}\sum_{\substack{h\leq q^2-n\\ h\equiv -n\moq}}\Lambda(h) \Big)+O\big(Y^{\frac{1}{2}+\varepsilon}\big)\\
:=&2\big(S_1(Y,n)- S_2(Y,n)\big)+O\big(Y^{\frac{1}{2}+\varepsilon}\big) .
\end{aligned}
\end{equation}
It follows from Lemma \ref{lem-bvprime} that
\begin{equation}\label{eq-s_1}
S_1(Y,n)=\sum_{\substack{q \leq \sqrt{Y+n} \\ (q,n)=1}}\frac{1}{\varphi(q)}\sum_{h\leq Y}\Lambda(h)+O\Big(\frac{Y}{(\log Y)^A}\Big),
\end{equation}
where the ranges of all parameters have been considered.  Now we turn to estimate $S_2(Y,n)$, and hope that $S_2(Y,n)$ can be approximated by
\[
\sum_{\substack{\sqrt{n}<q \leq \sqrt{Y+n}\\ (q,n)=1}} \frac{1}{\varphi(q)} \sum_{h\leq q^2-n}\Lambda(h).
\]
In fact, we shall show that the difference of $S_2(Y,n)$ and its approximation satisfies
\begin{equation}\label{eq-diff}
U(Y,n):=\sum_{\substack{\sqrt{n}<q \leq \sqrt{Y+n}\\ (q,n)=1}}\bigg(\sum_{\substack{h\leq q^2-n\\ h\equiv -n\moq}}\Lambda(h)- \frac{1}{\varphi(q)} \sum_{h\leq q^2-n}\Lambda(h)\bigg)\ll \frac{Y}{(\log Y)^A}.
\end{equation}
Since the condition $h\leq q^2-n$ makes $q$ and $h$ interconnected, Lemma \ref{lem-bvprime} can not be directly used.  To relax this condition, we shall decompose the sums over $q$ and $h$ into short intervals $\big(Q(1+\Delta_1)^{-1}, Q\big]$ and $\big(H(1+\Delta_1)^{-1}, H\big]$, where $\Delta_1=(\log Y)^{-B}$ with $B>A+3$.
Then we have
\begin{equation*}
\begin{aligned}
U(Y,n)\ll &\sum_{j_1,j_2\geq 0}\bigg|\sum_{\substack{Q(1+\Delta_1)^{-1}<q\leq Q \\ (q,n)=1}} \Big ( \sum_{\substack{H(1+\Delta_1)^{-1}<h\leq H\\h\leq q^2-n \\ n \equiv -n\moq}} \Lambda(h) - \frac{1}{\varphi(q)} \sum_{\substack{H(1+\Delta_1)^{-1}<h\leq H\\ h\leq q^2-n\\ (h,q) = 1}} \Lambda(h)\Big) \bigg|,
\end{aligned}
\end{equation*}
where $Q=\sqrt{Y+n}(1+\Delta_1)^{-j_1}, H=Y(1+\Delta_1)^{-j_2}$. Note that the condition $n \leq q^2$ can be dropped as soon as
$
H+n\leq Q^2(1+\Delta_1)^{-2}.
$
The contribution of $j_1, j_2$ such that $H+n> Q^2(1+\Delta_1)^{-2}$
is at most
\begin{equation*}
\begin{aligned}
\sum_{\substack{\sqrt{n}<q \leq \sqrt{Y+n}\\ (q,n)=1}} \; \sum_{\substack{(q^2-n)(1+\Delta_1)^{-3}<h\leq q^2-n \\ h \equiv -n\moq}}   \Lambda(h)\ll&  (\log Y)\sum_{\substack{\sqrt{n}<q \leq \sqrt{Y+n}\\ (q,n)=1}}\Big(\frac{\Delta_1(q^2-n)}{q} +1\Big)\\
\ll& \Delta_1Y(\log Y)^2.
\end{aligned}
\end{equation*}
Therefore, we obtain
\begin{equation*}
\begin{aligned}
U(Y,n)\ll &\Delta_1Y(\log Y)^2+\Delta_1^{-2}(\log Y)^2\\
&\times\max_{\substack{Q\leq \sqrt{Y+n}\\ H\leq Q
	^2-n}}\bigg|\sum_{Q(1+\Delta_1)^{-1}<q\leq Q \atop (q,n)=1} \Big ( \sum_{\substack{H(1+\Delta_1)^{-1}<h\leq H\\h \equiv -n\moq}} \Lambda(h) - \frac{1}{\varphi(q)} \sum_{\substack{H(1+\Delta_1)^{-1}<h\leq H\\  (h,q) = 1}} \Lambda(h)\Big) \bigg|.
\end{aligned}
\end{equation*}
Let $\delta>0$ be the real number given in Lemma \ref{lem-bvprime}. It is clear that
\[
\sum_{\substack{Q(1+\Delta_1)^{-1}<q\leq Q \\ (q,n)=1}} \Big ( \sum_{\substack{H(1+\Delta_1)^{-1}<h\leq H\\h \equiv -n\moq}} \Lambda(h) - \frac{1}{\varphi(q)} \sum_{\substack{H(1+\Delta_1)^{-1}<h\leq H\\  (h,q) = 1}} \Lambda(h)\Big)\ll \Delta_1H (\log H),
\]
which is acceptable if $H\leq \Delta_1^2\,Y$. Suppose $H> \Delta_1^2\,Y$, then Lemma \ref{lem-bvprime} can be applied and yields
\[
\sum_{\substack{Q(1+\Delta_1)^{-1}<q\leq Q \\ (q,n)=1}} \Big ( \sum_{\substack{H(1+\Delta_1)^{-1}<h\leq H\\n \equiv -n\moq}} \Lambda(h) - \frac{1}{\varphi(q)} \sum_{\substack{H(1+\Delta_1)^{-1}<n\leq H\\  (n,q) = 1}} \Lambda(h)\Big)\ll \frac{H}{(\log H)^{3A}}.
\]
In summary, we have the claimed estimate \eqref{eq-diff}, which means
\begin{equation}\label{eq-s_2}
S_2(Y,n)=
\sum_{\substack{\sqrt{n}<q \leq \sqrt{Y+n}\\(q,n)=1}} \frac{1}{\varphi(q)} \sum_{h\leq q^2-n}\Lambda(h)+O\Big(\frac{Y}{(\log Y)^A}\Big).
\end{equation}
Inserting \eqref{eq-s_1} and \eqref{eq-s_2} into \eqref{eq-TYn}, and applying Lemma \ref{lem-phisum} and partial summation, we have
\begin{equation*}
\begin{aligned}
T(Y,n)=& 2\sum_{h\leq Y}\Lambda(h) \sum_{\substack{q \leq \sqrt{h+n}\\ (q,n)=1}}\frac{1}{\varphi(q)}+O\Big(\frac{Y}{(\log Y)^A}\Big)\\
=& c_{0}(n) \big((Y+n)\log(Y+n)-n\log n+(2\gamma -1)Y\big)+2c_{0}^{\prime}(n)Y+O\Big(\frac{Y}{(\log Y)^A}\Big),
\end{aligned}
\end{equation*}
where $Y=X/(\log X)^A$. Substituting this asymptotic formula into \eqref{eq-1top}, we then get
\begin{equation*}
\begin{aligned}
S_{f}(X,-1)=\sum_{n \leq X}f(n) \big(c_{0}(n)(\log n+2\gamma)+2c_{0}^{\prime}(n)\big)+E_{f}(X,-1)+O\big(X(\log X)^{\frac{1}{2}+\varepsilon}\big),
\end{aligned}
\end{equation*}
where the term $E_{f}(X,-1)$ satisfies
\[
E_{f}(X,-1)\ll \sum_{n\leq X} |f(n)c_{0}(n)| \log \big(1+\frac{Y}{n}\big)+\sum_{n\leq X} |f(n)c_{0}(n)|\,\Big|\frac{n}{Y}\log \big(1+\frac{Y}{n}\big) -1\Big|.
\]
Notice that $c_{0}(n)\ll 1$. It follows from partial summation and Hypothesis (i) that
\[
E_{f}(X,-1)\ll  \frac{X}{(\log X)^{A-\frac{c+1}{2}}}.
\]
Hence, taking $A> c/ 2$ gives
\[
S_{f}(X,-1)=\sum_{n \leq X}f(n) \big(c_{0}(n)(\log n+2\gamma)+2c_{0}^{\prime}(n)\big)+O\big(X(\log X)^{\frac{1}{2}+\varepsilon}\big).
\]
This completes the proof of Theorem \ref{thm-main}.

\section{Proof of Theorem \ref{thm-BV}}

For convenience, we put
\begin{equation*}
E (f, X; q) := \sum_{\substack{n\leq X \\ n\equiv 1 \moq}} f(n)-\frac{1}{\varphi (q) } \sum_{\substack{n\leq X \\ (n,q) =1}} f(n).
\end{equation*}
Let $c_q$ be given by $c_q=\sgn\, E (f, X; q)$. With the notation as in \eqref{eq-partition-IJ}, we divide $E (f, X; q)$  into two parts  $E^{\I}(f, X; q)$ and $E^{\J} (f, X; q)$, which are the corresponding sums with additional restrictions $n\in \I$ and $n\in \J$, respectively.

\subsection{Contribution of $E^{\I} (f, X; q)$}
Considering the $q$ in dyadic ranges, it suffices to estimate
$
\sum_{q\sim Q}c_qE^{\I} (f, X; q)
$
with $Q\leq X^{17/33-\varepsilon}.$ By the definition of interval $\I$ and using the  Cauchy--Schwarz inequality, we have
\begin{equation}\label{eq-ap-I}
\begin{aligned}
\sum_{q\sim Q}c_qE^{\I} (f, X; q)=&\sum_{0\leq \nu\leq L} \sum_{m\in\M_\nu}f(m)\Big( \underset{\substack{q \sim Q\ p\in\P_\nu \\ pm \equiv 1 \moq}}{\sum\  \sum } c_q f(p)  -
\underset{\substack{q \sim Q\ p\in\P_\nu  \\ (pm, q)=1}}{\sum\  \sum } \frac{c_qf(p)}{\varphi (q)}  \Big)  \\
\ll&\sum_{0\leq \nu\leq L} \Big(\sum_{m\in\M_\nu} |f(m)|^2\Big)^{\frac{1}{2}} T_{\nu}(Q)^{\frac{1}{2}},
\end{aligned}
\end{equation}
where $T_{\nu}(Q)$ is defined by
\[
T_{\nu}(Q)=\sum_{m\leq X/H_{\nu + 1}}  \Big| \underset{\substack{q \sim Q\ p\in\P_\nu \\ pm \equiv 1 \moq}}{\sum\  \sum } c_q f(p)  -
\underset{\substack{q \sim Q\ p\in\P_\nu  \\ (pm, q)=1}}{\sum\  \sum } \frac{c_qf(p)}{\varphi (q)} \Big|^2.
\]
By smooth dyadic subdivisions and squaring out $T_{\nu}(Q)$, we then write
\begin{equation}\label{eq-Tv}
\begin{aligned}
&T_{\nu}(Q)\ll (\log X)\max_{M\leq X/H_{\nu + 1}}\sum_{m}\psi\big(\frac{m}{M}\big)  \Big| \underset{\substack{q \sim Q\ p\in\P_\nu \\ pm \equiv 1 \moq}}{\sum\  \sum } c_q f(p)  -
\underset{\substack{q \sim Q\ p\in\P_\nu  \\ (pm, q)=1}}{\sum\  \sum } \frac{c_qf(p)}{\varphi (q)}\Big|^2\\
=&(\log X)\max_{M\in[X^{3/4},X/H_{\nu + 1}]}\big(W_{\nu}(Q,M)-2\Re V_{\nu}(Q,M)+U_{\nu}(Q,M) \big) +O\big(X^{\frac{3}{4}+\varepsilon}\big),
\end{aligned}
\end{equation}
where $\psi$ is the smooth function as in Lemma \ref{lem-poisson}, and $W_{\nu}(Q,M), V_{\nu}(Q,M), U_{\nu}(Q,M)$ are defined by
\begin{equation*}
\begin{aligned}
W_{\nu}(Q,M)=&\sum_{m}\psi\big(\frac{m}{M}\big)  \Big| \underset{\substack{q \sim Q\ p\in\P_\nu \\ pm \equiv 1 \moq}}{\sum\  \sum } c_q f(p)  \Big|^2,\\
V_{\nu}(Q,M)=&\sum_{m}\psi\big(\frac{m}{M}\big)  \, \Big(\underset{\substack{q \sim Q\ p\in\P_\nu \\ pm \equiv 1 \moq}}{\sum\  \sum } c_q f(p) \Big)  \,\Big(
\underset{\substack{q \sim Q\ p\in\P_\nu  \\ (pm, q)=1}}{\sum\  \sum } \frac{\overline{c_qf(p)}}{\varphi (q)} \Big),\\
U_{\nu}(Q,M)=&\sum_{m}\psi\big(\frac{m}{M}\big)  \Big|
\underset{\substack{q \sim Q\ p\in\P_\nu  \\ (pm, q)=1}}{\sum\  \sum } \frac{c_q f(p)}{\varphi (q)} \Big|^2.
\end{aligned}
\end{equation*}
Our goal is to evaluate each term individually.

\subsection{Evaluation of $U_{\nu}(Q,M)$} Let us begin with the simplest term
\[
U_{\nu}(Q,M)=\sum_{q_1\sim Q} \sum_{q_2\sim Q} \frac{c_{q_1} \overline{c_{q_2}}}{\varphi (q_1)\varphi (q_2)}  \sum_{\substack{p_1\in\P_\nu\\ (p_1, q_1)=1}}
\sum_{\substack{p_2\in\P_\nu \\ (p_2, q_2)=1}} f(p_1)\overline{f(p_2)} \sum_{(m, q_1q_2)=1} \psi \Big( \frac{m}{M}\Big).
\]
By Poisson's formula \eqref{eq-poi-p} in Lemma \ref{lem-poisson}, the innermost sum over $m$ is
\[
\sum_{(m, q_1q_2)=1} \psi \Big( \frac{m}{M}\Big)=\frac{\varphi (q_1q_2)}{q_1q_2}\hat{\psi}(0)  M + O \big(\tau (q_1) \tau (q_2)(\log X)^4 \big).
\]
This yields
\[
U_{\nu}(Q,M)=\hat{\psi}(0)  M \sum_{q_1\sim Q} \sum_{q_2\sim Q} \frac{c_{q_1} \overline{c_{q_2}}}{\varphi (q_1)\varphi (q_2)}\, \frac{\varphi (q_1q_2)}{q_1q_2} \sum_{\substack{p_1\in\P_\nu \\ (p_1, q_1)=1}}
\sum_{\substack{p_2\in\P_\nu \\ (p_2, q_2)=1}} f(p_1)\overline{f(p_2)}+O \Big(\frac{H_v^2}{(\log X)^{\frac{A}{2}-8}}\Big).
\]
We remove the conditions $(p_1, q_1)=1$ and $(p_2, q_2)=1$ at the cost
of the admissible error term $O(MH_{\nu})$. Moreover, we denote $\delta = (q_1,q_2)$, and then get
\begin{equation}\label{eq-U-estimate}
\begin{aligned}
U_{\nu}(Q,M)=&\hat{\psi}(0)  M\sum_\delta \frac{1}{\delta \varphi (\delta)} \underset{\substack{k_1,  k_2 \sim Q/\delta \\ (k_1,k_2)=1}}{\sum\ \sum} \frac{ c_{\delta k_1} \overline{c_{\delta k_2}}}{k_1k_2} \Bigl( \sum_{p_1\in\P_\nu }f(p_1) \Bigr) \, \Bigl( \sum_{p_2\in\P_\nu } \overline{f(p_2)}\Bigr)\\
&+O(MH_{\nu})+O \Big(\frac{H_v^2}{(\log X)^{\frac{A}{2}-8}}\Big).
\end{aligned}
\end{equation}

\vskip 2mm

\subsection{Evaluation of $V_{\nu}(Q,M)$}
By the definition of $V_{\nu}(Q,M)$ and exchanging the order of summations, we have
\begin{equation}\label{eq-V-1}
V_{\nu}(Q,M)=  \sum_{q_1\sim Q} \sum_{q_2\sim Q} \frac{c_{q_1} \overline{c_{q_2}}}{\varphi (q_2)}  \sum_{\substack{p_1\in\P_\nu\\ (p_1, q_1)=1}}
\sum_{\substack{p_2\in\P_\nu \\ (p_2, q_2)=1}} f(p_1)\overline{f(p_2)}  \sum_{\substack{m\equiv \overline{p_1}\modq \\ (m,q_2) =1}} \psi \Bigl( \frac{m}{M}\Bigr),
\end{equation}
where $\overline{p_1}$ denotes the multiplicative inverse of $p_1$ modulo $q_1$.   By applying the M\" obius inversion formula and inserting Poisson's formula \eqref{eq-poi} with $H=dq_1(\log X)^4/M$, the innermost sum is equal to
\[
\begin{aligned}
\sum_{\substack{m\equiv \overline{p_1}\modq \\ (m,q_2) =1}} \psi \Bigl( \frac{m}{M}\Bigr)=& \sum_{\substack{d|q_2 \\ (d,q_1)=1}}\mu(d)\sum_{\substack{m\equiv \eta \moddq }} \psi \Bigl( \frac{m}{M}\Bigr)\\
=& \frac{\hat{\psi}(0)  M}{q_1}\sum_{\substack{d|q_2 \\ (d,q_1)=1}}\frac{\mu(d)}{d} +O \big( \tau(q_2) (\log X)^4 \big),
\end{aligned}
\]
where $\eta$ is a common solution of the congruences $\eta \equiv \overline{p_1} \modq$  and $\eta \equiv 0 \modd$. Inserting this into \eqref{eq-V-1}, we deduce that
\begin{equation*}
\begin{aligned}
V_{\nu}(Q,M)=&\hat{\psi}(0)  M\sum_\delta \frac{1}{\delta \varphi (\delta)} \underset{\substack{r_1,  r_2 \sim Q/\delta \\ (r_1,r_2)=1}}{\sum\ \sum} \frac{ c_{\delta r_1} \overline{c_{\delta r_2}}}{r_1r_2} \Bigl( \sum_{\substack{p_1\in\P_\nu \\ (p_1, q_1)=1}}f(p_1) \Bigr) \, \Bigl( \sum_{\substack{p_2\in\P_\nu \\ (p_2, q_2)=1}} \overline{f(p_2)}\Bigr)\\
&+O \Big(\frac{QH_v^2}{(\log X)^{\frac{A}{2}-6}}\Big).
\end{aligned}
\end{equation*}
Comparing this with \eqref{eq-U-estimate}, we obtain the relation
\begin{equation}\label{eq-V-estimate}
V_{\nu}(Q,M)=U_{\nu}(Q,M)+O(MH_{\nu})+O \Big(\frac{QH_v^2}{(\log X)^{\frac{A}{2}-8}}\Big).
\end{equation}
\vskip 2mm

\subsection{Evaluation of $W_{\nu}(Q,M)$} The evaluation of $W_{\nu}(Q,M)$ is the most difficult and it involves the key arguments. Before
applying them, in this section we reduce the range of the summation by elementary estimates. By definition, we have
\begin{equation}\label{eq-W-1}
W_{\nu}(Q,M)=  \sum_{q_1\sim Q} \sum_{q_2\sim Q} c_{q_1} \overline{c_{q_2}} \sum_{\substack{p_1,p_2\in\P_\nu\\ p_1\neq p_2}}
f(p_1)\overline{f(p_2)}  \sum_{\substack{p_1m\equiv 1 \modq \\ p_2m\equiv 1 \motq}} \psi \Bigl( \frac{m}{M}\Bigr)+O(H_{\nu}M\log X),
\end{equation}
where the error term comes from the diagonal contribution for $p_1=p_2$. For the off-diagonal terms, we first factorize each of the variables $q_1$ and $q_2$, to control the contributions of their greatest common divisor in some ranges. For notational conventions, we decompose the variables $q_1, q_2$ in a unique way as follows:
\begin{equation*}\label{conv}
\begin{cases}
\delta = (q_1, q_2),\\
q_1 = \delta k_1, \, q_2 = \delta k_2,\\
k_1= \delta_1 k'_1 \text{ with } \delta_1 \vert \delta^\infty \text{  and } (k'_1 ,\delta) =1,\\
k_2= \delta_2 k'_2 \text{ with } \delta_2 \vert \delta^\infty \text{  and } (k'_2 ,\delta) =1.
\end{cases}
\end{equation*}
Note that the conditions of summation over $m$ in \eqref{eq-W-1} imply that we necessarily have
\begin{equation*}\label{n1congn2}
p_1 \equiv p_2 \moddelta,
\end{equation*}

Denote $W_{\nu}^{(1)}(Q,M)$ to be the sum of the terms in \eqref{eq-W-1} with $\delta> D:=\exp\big((\log X)^{\varepsilon/4}\big)$. Thus, we obtain from Lemma \ref{lem-BT} that
\begin{equation}\label{eq-delta-L}
\begin{aligned}
W_{\nu}^{(1)}(Q,M)\ll &\sum_{\delta > X^{\varepsilon}}  \sum_{k_2 \sim Q/\delta}\sum_{\substack{p_1,p_2\in\P_\nu\\ p_1 \equiv p_2 \moddelta}}
|f(p_1)|^2
\sum_{\substack{M/2 < m <5 M/2  \\ m\equiv  \overline{p_2} (\text{mod}\, \delta  k_2)}} \tau (\vert p_1m-a\vert )  \\
\ll &\frac{H_{\nu}^2 M(\log X)^4}{D}.
\end{aligned}
\end{equation}
Denote $W_{\nu}^{(2)}(Q,M)$ to be the sum of terms in \eqref{eq-W-1} with $\delta\leq D$ and $\delta_1> D$. Then we obtain
\begin{equation}\label{eq-delta1-L}
\begin{aligned}
W_{\nu}^{(2)}(Q,M)\ll &  \sum_{\delta \leq D} \sum_{\substack{\delta_1 \vert \delta^{\infty} \\ \delta_1 > X^{\varepsilon}}} \ \sum_{k'_1 \sim Q/(\delta \delta_1)} \sum_{\substack{p_1,p_2\in\P_\nu\\ p_1 \equiv p_2 \moddelta}}
|f(p_1)|^2
\sum_{\substack{M/2 < m <5 M/2  \\ m\equiv \overline{p_2} (\text{mod}\,\delta  \delta_1 k'_1) }} \tau (\vert p_2m-a\vert )  \\
\ll &\frac{H_{\nu}^2 M}{D^{\frac{1}{2}}}.
\end{aligned}
\end{equation}
In the same manner, the sum of of terms in \eqref{eq-W-1} with  $\delta\leq D$ and $\delta_2>D$ is also bounded by $O\big(H_{\nu}^2 M/D^{1/2}\big)$.

Now, we define $W_{\nu}^{(0)}(Q,M)$ to be the partial sum on the left-hand side of \eqref{eq-W-1} restricted by
$
\delta\leq D,  \delta_1\leq D, \delta_2\leq D.
$
Then we get from the discussion before that
\begin{equation*}\label{eq-w-w0}
W_{\nu}(Q,M)=W_{\nu}^{(0)}(Q,M)+O\Big(\frac{H_{\nu}^2 M}{D^{\frac{1}{2}}}\Big),
\end{equation*}
where $W_{\nu}^{(0)}(Q,M)$ can be written as
\begin{equation}\label{eq-W-1-0}
\sum_{\delta,  \delta_1, \delta_2\leq D} \sum_{ k'_1 \sim Q/(\delta \delta_1)\ }
\sum_{ k'_2 \sim Q /( \delta \delta_2)}c_{\delta \delta_1k'_1} \overline{c_{\delta \delta_2 k'_2}}
\sum_{\substack{p_1,p_2\in\P_\nu\\ p_1\neq p_2}}
f(p_1)\overline{f(p_2)}  \sum_{\substack{m\equiv \overline{p_1}(\text{mod}\,\delta \delta_1k'_1) \\ m\equiv \overline{p_2}(\text{mod}\,\delta \delta_2k'_2)}} \psi \Bigl( \frac{m}{M}\Bigr).
\end{equation}
The congruence conditions in \eqref{eq-W-1-0} are equivalent to the four congruences
\begin{equation*}\label{eq-cong-1}
\begin{aligned}
&m\equiv  \overline {p_1} (\text{mod}\,\delta \delta_1),   \quad m\equiv  \overline {p_2} (\text{mod}\,\delta \delta_2),\\
&  m\equiv   \overline {p_1} (\text{mod}\,k'_1),   \quad\;\, m\equiv  \overline {p_2} (\text{mod}\,k'_2).
\end{aligned}
\end{equation*}
The first two equations are equivalent to $m\equiv a \lambda (\text{mod}\, \delta \delta_1 \delta_2)$, where $\lambda (p_1, p_2)$ is some congruence class modulo $\delta \delta_1 \delta_2$, only depending on the congruence classes of $p_1 \bmod \delta \delta_1$ and $p_2 \bmod \delta \delta_2$. Finally, we see that $m$ satisfies the single congruence
\begin{equation*}
m\equiv m_0(\text{mod}\, \eta k'_1 k'_2),
\end{equation*}
with
\begin{equation*}
\eta = \delta \delta_1 \delta_2
\end{equation*}
and
$$
m_0 =  \lambda \,k'_1k'_2\, \overline{ k'_1} \, \overline{k'_2}
+\eta \,\overline{\eta } \,\overline{p_1}\,k'_2 \, \overline{k'_2}  +\, \eta\, \overline{\eta }\, \overline{p_2} \, k'_1 \overline{k'_1}
(\text{mod}\, \eta k'_1 k'_2) ,
$$
where the $\overline{x}$--symbol  respectively means the inverse of $x$  modulo $\eta $, $k'_{1}$ and $k'_{2}$.

We then apply Poisson's formula in Lemma \ref{lem-poisson} and deduce
\begin{equation}\label{eq-ap-poi-W}
\sum_{m\equiv m_0(\text{mod}\, \eta k'_1 k'_2)} \psi \Bigl( \frac{m}{M}\Bigr)=\hat{\psi}(0)  \frac {M}{\eta k'_1 k'_2}
+ \frac{M}{\eta k'_1 k'_2} \sum_{1 \leq \vert h \vert \leq H}e \bigl( \frac{ hm_0}{\eta k'_1 k'_2} \bigr)\hat \psi \Bigl( \frac{hM}{\eta k'_1 k'_2}\Bigr)+O (M^{-1})
\end{equation}
with
\begin{equation*}
H=  M^{-1}Q^2 (\log X)^4.
\end{equation*}
Inserting \eqref{eq-ap-poi-W} into \eqref{eq-W-1-0}, we obtain
\begin{equation}\label{eq-splitW}
W_{\nu}^{(0)}(Q,M)= \hat{\psi}(0)M \R + M \R_1+ O\Big(\frac{H_{\nu}^2 Q^2}{M}\Big),
\end{equation}
where
\begin{equation*}\label{eq-def-R}
\R =\sum_{\delta,  \delta_1, \delta_2\leq D} \sum_{ k'_1 \sim Q/(\delta \delta_1)\ }
\sum_{ k'_2 \sim Q /( \delta \delta_2)}\frac{c_{\delta \delta_1k'_1} \overline{c_{\delta \delta_2 k'_2}}}{\delta \delta_1 \delta_2 k'_1k'_2}
\sum_{\substack{p_1,p_2\in\P_\nu\\ p_1 \equiv p_2 \moddelta}}
f(p_1)\overline{f(p_2)},
\end{equation*}

\begin{equation*}\label{eq-def-R1}
\begin{aligned}
\R_1 =&\sum_{\delta,  \delta_1, \delta_2\leq D} \sum_{ k'_1 \sim Q/(\delta \delta_1)\ }
\sum_{ k'_2 \sim Q /( \delta \delta_2)}\frac{c_{\delta \delta_1k'_1} \overline{c_{\delta \delta_2 k'_2}}}{\eta k'_1k'_2}
\sum_{\substack{p_1,p_2\in\P_\nu\\ p_1 \equiv p_2 \moddelta}}
f(p_1)\overline{f(p_2)}  \\
&\times \sum_{1 \leq \vert h \vert \leq H}e \bigl( \frac{ hm_0}{\eta k'_1 k'_2} \bigr)\hat \psi \Bigl( \frac{hM}{\eta k'_1 k'_2}\Bigr)
\end{aligned}
\end{equation*}
and the error term comes from the contribution of the term $O (M^{-1})$ in \eqref{eq-ap-poi-W}. Now it remains to evaluate $\R$ and $\R_1$.

We first treat the term $\R$, which can be rewritten as
\[
\R=\sum_{\delta,  \delta_1, \delta_2\leq D} \sum_{ k'_1 \sim Q/(\delta \delta_1)\ }
\sum_{ k'_2 \sim Q /( \delta \delta_2)}\frac{c_{\delta \delta_1k'_1} \overline{c_{\delta \delta_2 k'_2}}}{\delta \delta_1 \delta_2 k'_1k'_2} \sideset{}{^*}\sum_{\alpha\moddelta }\Big(\sum_{\substack{p_1\in\P_\nu\\ p_1 \equiv \alpha \moddelta}}
f(p_1)\Big) \Big(\sum_{\substack{p_2\in\P_\nu\\ p_2 \equiv \alpha \moddelta}}
\overline{f(p_2)}\Big).
\]
By Hypotheses (ii) and (iii), the Barban--Davenport--Halberstam theorem (see \cite[Theorem 17.5]{IK}) yields
\begin{equation}\label{eq-BDH}
\sum_{k\leq K}\;  \sideset{}{^*}\sum_{\alpha (\text{mod}\, k)}\Big|\sum_{\substack{p\in\P_\nu\\ p \equiv \alpha (\text{mod}\, k)}}
f(p)- \frac{1}{\varphi(k)} \sum_{p\in\P_\nu}
f(p)\Big|^2\ll \frac{H_{\nu}^2}{(\log X)^{16A}}
\end{equation}
provided that $K\leq\exp\big((\log X)^{\varepsilon/2}\big)$. Then it follows from \eqref{eq-BDH} that
\[
\R =\R_0+O\Big(\frac{H_{\nu}^2}{(\log X)^{8A}} \Big),
\]
where $\R_0$ is given by
\[
\R_0=\sum_{\delta,  \delta_1, \delta_2\leq D} \sum_{ k'_1 \sim Q/(\delta \delta_1)\ }
\sum_{ k'_2 \sim Q /( \delta \delta_2)}\frac{c_{\delta \delta_1k'_1} \overline{c_{\delta \delta_2 k'_2}}}{\delta \delta_1 \delta_2 k'_1k'_2\varphi(\delta)} \Big(\sum_{p_1\in\P_\nu}
f(p_1)\Big) \Big(\sum_{p_2\in\P_\nu}
\overline{f(p_2)}\Big).
\]
We next extend  the summation over all $\delta,  \delta_1, \delta_2$  to get
\[
\R_0 =\R^{*}+O\big(D^{-\frac{1}{2}}H_{\nu}^2 \big)
\]
the error term being estimated by similar arguments as those for \eqref{eq-delta-L} and  \eqref{eq-delta1-L}, where $\R^{*}$ is given by
\[
\begin{aligned}
\R^{*}=&\sum_{\delta,  \delta_1, \delta_2} \ \sum_{ k'_1 \sim Q/(\delta \delta_1)\ }
\sum_{ k'_2 \sim Q /( \delta \delta_2)}\frac{c_{\delta \delta_1k'_1} \overline{c_{\delta \delta_2 k'_2}}}{\delta \delta_1 \delta_2 k'_1k'_2\varphi(\delta)} \Big(\sum_{p_1\in\P_\nu}
f(p_1)\Big) \Big(\sum_{p_2\in\P_\nu}
\overline{f(p_2)}\Big)\\
=&\sum_{\delta } \frac{1}{\delta\varphi(\delta)}  \underset{\substack{k_{1},k_2 \sim Q/\delta   \\ (k_1, k_2) = 1}}{\sum \sum}
\frac{c_{\delta k_{1}} \overline{c_{\delta k_{2}}}}{\delta  k_1k_2} \Big(\sum_{p_1\in\P_\nu}
f(p_1)\Big) \Big(\sum_{p_2\in\P_\nu}
\overline{f(p_2)}\Big).
\end{aligned}
\]
So we can conclude that
\begin{equation}\label{eq-R-finally}
\R =\R^{*}+O\Big(\frac{H_{\nu}^2}{(\log X)^{16A}} \Big).
\end{equation}

We are left to estimate $\R_1.$ By definition of $m_0,$ we know
\[
e \bigl( \frac{ hm_0}{\eta k'_1 k'_2} \bigr)=e\Bigl(\lambda h \frac{\overline{k'_1}\, \overline{ k'_2}}{\eta}
+h \frac{ \overline \eta \, \overline {p_1}\,  \overline{k'_2}}{k'_1}
+  h \frac{ \overline \eta\,  \overline{p_2} \, \overline{k'_1}}{k'_2}
\Bigr).
\]
In order to transform the exponent, we apply Bezout's relation twice to write
\begin{align*}
\frac{ \overline \eta \, \overline {p_1}\,  \overline{k'_2}}{k'_1}
&= \frac{1}{\eta p_1 k'_1 k'_2}
- \frac{\overline{k'_1}}{ \eta
	p_1  k'_2}
\,(\text{mod}\,1) \\
& = \frac{1}{\eta p_1 k'_1 k'_2}
- \frac{\overline{ \eta} \,   \overline{ k'_1}}{p_1 k'_2}
- \frac{\overline{p_1} \,\overline{k'_1} \, \overline{k'_2}}{ \eta } \,(\text{mod}\,1).
\end{align*}
Moreover, we have
\[
\frac{ \overline \eta\,  \overline{p_2} \, \overline{k'_1}}{k'_2} - \frac{\overline{ \eta} \,   \overline{ k'_1}}{p_1 k'_2} =\frac{\overline{ \eta} (p_1-p_2)\overline{p_2k'_1}  }{p_1 k'_2} \,(\text{mod}\,1).
\]
Thus, these relations yield
\[
e \bigl( \frac{ hm_0}{\eta k'_1 k'_2} \bigr)=e\Bigl(\lambda h \frac{\overline{k'_1}\, \overline{ k'_2}}{\eta} -h\frac{\overline{p_1} \,\overline{k'_1} \, \overline{k'_2}}{ \eta }
+  \frac{h}{\eta p_1 k'_1 k'_2}
+  h \frac{\overline{ \eta} (p_1-p_2)\overline{p_2k'_1}  }{p_1 k'_2}
\Bigr).
\]
Next we split up the summation over  $p_1, p_2, h, k'_1, k'_2$ into  arithmetic progressions $\alpha_i (\text{mod}\, \eta), 1\leq i\leq 5$. If we fix the congruence class $\alpha_i (\text{mod}\, \eta)$, then the exponents  $\lambda h \frac{\overline{k'_1}\, \overline{ k'_2}}{\eta}$ and $h\frac{\overline{p_1} \,\overline{k'_1} \, \overline{k'_2}}{ \eta }$ are fixed. Thus, we further have
\begin{equation*}\label{eq-R1-3}
\begin{aligned}
\R_1 \ll D^{18} \max_{\alpha_i\leq D^3 \atop \delta,  \delta_1, \delta_2\leq D}\Big|& \sum_{\substack{p_1,p_2\in\P_\nu\\ p_1 \neq p_2}}\sum_{1 \leq |h|\leq H}  \sum_{ k'_1 \sim Q/(\delta \delta_1)\ }
\sum_{ k'_2 \sim Q /( \delta \delta_2)} \frac{c_{\delta \delta_1k'_1} \overline{c_{\delta \delta_2 k'_2}}}{ k'_1k'_2}f(p_1)\overline{f(p_2)} \\
&\times \hat \psi \Bigl( \frac{hM}{\eta k'_1 k'_2}\Bigr)
e\Bigl( \frac{h }{\eta p_1 k'_1 k'_2}
\Bigr)\,  e\Bigl(h \frac{\overline{ \eta} (p_1-p_2)\overline{p_2k'_1}  }{p_1 k'_2}
\Bigr)\Big|,
\end{aligned}
\end{equation*}
where the variables $p_1, p_2$ satisfy the congruence condition $p_1 \equiv p_2 \moddelta$, and  $p_1, p_2, h, k'_1, k'_2$ are in congruence classes $\alpha_i (\text{mod}\, \eta), 1\leq i\leq 5$, respectively.

A straightforward calculation gives
$$
\frac{ \partial^{a_1+a_2+a_3+a_4}}{\partial h^{a_1}\partial p_{1}^{a_2} \partial {k'}_1^{a_3} \partial {k'}_2^{a_4}}\Bigl\{e\Bigl( \frac{h }{\eta p_1 k'_1 k'_2}
\Bigr)\hat \psi \Bigl( \frac{hM}{\gamma k'_1 k'_2}\Bigr)\Bigr\}
\ll (1 + \vert h\vert)^{-a_1} p_{1}^{-a_2} {k'}_1^{-a_3} {k'}_2^{-a_4}
$$
for integers $0 \leq a_1, a_2, a_3 \leq 1$, $p_1\in \P_\nu, k_1', k_2' \in [Q / D^{2}, 2Q]$ and $h \in [-H, H]$. By partial summations over variables $p_1, h, k_1', k_2'$,  we obtain
\begin{equation}\label{eq-R1-4}
\begin{aligned}
\R_1 \ll & \frac{D^{22}\log X }{Q^{2}}\max_{\delta,  \delta_1, \delta_2\leq D}\sum_{\substack{p_1,p_2\in\P_\nu\\ p_1 \neq p_2}}\big|f(p_1)f (p_2)\big|\\
&\times \Big|\sum_{1 \leq |h|\leq H} \sum_{ k'_1,k'_2 \leq 2Q}  \xi_1 (h) \xi_2 (k'_1) \xi_3 (k'_2) e\Bigl(h \frac{\overline{ \eta} (p_1-p_2)\overline{p_2k'_1}  }{p_1 k'_2}
\Bigr)\Big|,
\end{aligned}
\end{equation}
where $\xi_1, \xi_2, \xi_3$ are three sequences of complex numbers with $|\xi_1 (n)|, |\xi_2(n)|, |\xi_3 (n)|\leq 1$.  Notice that the symmetry of variables $p_1, p_2$ allows us to without loss of generality replace the condition $1 \leq |h| \leq H$ by $1 \leq h \leq H$.

We shall apply Lemma \ref{lem-trilinear}  to bound the exponential sum on the second line of \eqref{eq-R1-4}. So we localize each of the variables $h, k'_1, k'_2$ dyadically around powers of two that we denote respectively by $H_{1}, K_1, K_2$.
For each such dyadic partition and fixed variables $p_1, p_2 \in \P_\nu, \eta\leq D^3$, we apply Lemma \ref{lem-trilinear} with the following choice of variables:
$$
\vartheta \rightarrow p_1-p_2, \; a\rightarrow h, \; m\rightarrow \eta p_2 k'_1,\; n\rightarrow p_1 k'_2,
$$
where the left side of $\rightarrow$ corresponds to notations of Lemma \ref{lem-trilinear} while the right side of $\rightarrow$ corresponds to our current notation)
and parameters
$$
\vert \vartheta \vert \ll H_{\nu}, \ A\rightarrow  H_1, \  M \rightarrow \eta p_2 K_1, \ N \rightarrow p_1 K_2.
$$
Note that $\eta p_2 K_1 \ll D^{3} H_{\nu} K_1$ and that $p_1 K_2 \ll H_{\nu} K_2$. Hence, we appeal to Lemma \ref{lem-trilinear} to infer
\[
\begin{aligned}
\sum_{h\sim H_1} \sum_{ k'_1\sim K_1}&\sum_{ k'_2\sim K_2}  \xi_1 (h) \xi_2 (k'_1) \xi_3 (k'_2) e\Bigl(h \frac{\overline{ \eta} (p_1-p_2)\overline{p_2k'_1}  }{p_1 k'_2}
\Bigr)
\ll
D^2(H_{1}K_1 K_2 )^\frac{1}{2} \Bigl( 1 + \frac{ H_{1}}{H_{\nu} K_1 K_2} \Bigr)^\frac{1}{2} \\
&\times\Big( ( H_{1} H_{\nu}^2 K_1 K_2)^{\frac{7}{20} +\varepsilon} (H_{\nu}(K_1 + K_2))^\frac{1}{4} + ( H_{1} H_{\nu}^2 K_1 K_2)^{\frac{3}{8} +\varepsilon} (H_{1} H_{\nu}(K_1 + K_2)  )^\frac{1}{8}
\Big).
\end{aligned}
\]
Summing over all the dyadic partitions and applying Hypothesis (ii), we then get
\begin{equation}\label{eq-estimate-R1}
	\R_1 \ll  D^{25} \bigl( M^{-\frac{17}{20}}H_{\nu}^\frac{59}{20}  Q^\frac{33}{20}+M^{-1} H_{\nu}^\frac{23}{8} Q^\frac{15}{8}\bigr).
\end{equation}

Finally, inserting the estimates \eqref{eq-R-finally}, \eqref{eq-estimate-R1} of $\R$ and $\R_1$ into \eqref{eq-splitW}, we have
\begin{equation*}
W_{\nu}(Q,M)= \hat{\psi}(0)M \R^{*}+ O\Big( M^{\frac{3}{20}+\varepsilon}H_{\nu}^\frac{59}{20}  Q^\frac{33}{20}+M^{\varepsilon}H_{\nu}^\frac{23}{8} Q^\frac{15}{8}+\frac{MH_{\nu}^2}{(\log X)^{8A}}\Big).
\end{equation*}
Comparing this with \eqref{eq-U-estimate}, we also have the relation
\begin{equation}\label{eq-W-together}
W_{\nu}(Q,M)=U_{\nu}(Q,M)+ O\Big( M^{\frac{3}{20}+\varepsilon}H_{\nu}^\frac{59}{20}  Q^\frac{33}{20}+M^{\varepsilon}H_{\nu}^\frac{23}{8} Q^\frac{15}{8}+\frac{MH_{\nu}^2}{(\log X)^{8A}}\Big).
\end{equation}

\vskip 2mm

\subsection{Return Section 8.1}
We recall that $Q\leq X^{17/33-\varepsilon}.$ Substituting \eqref{eq-U-estimate},  \eqref{eq-V-estimate} and   \eqref{eq-W-together} into \eqref{eq-Tv},  we see
\begin{equation*}\label{defT(Q)}
T_{\nu}(Q)\ll \frac{XH_{\nu}}{(\log X)^{4A}},
\end{equation*}
which further yields from \eqref{eq-ap-I} that
\begin{equation}\label{eq-ap-I-F}
\sum_{q\leq  Q}c_qE^{\I} (f, X; q)\ll \frac{X}{(\log X)^{A-2}}
\end{equation}
holds for any $A\geq 2$.

\vskip 5mm

\subsection{Contribution of $E^{\J} (f, X; q)$}
With the notation as in Section 6, we have
\begin{equation}\label{eq-J-trans}
\begin{aligned}
\sum_{q\leq  Q}c_qE^{\J} (f, X; q)\ll  & \sum_{n\in \J} |f(n)| \sum_{q\leq Q\atop n \equiv 1 \moq}|c_q|+(\log Q) \sum_{n\in \J} |f(n)|\\
\ll  & \sum_{n\in \J} |f(n)| \tau(n-1)+(\log X) \sum_{n\in \J} |f(n)|
\end{aligned}
\end{equation}
The first term on the second line of \eqref{eq-J-trans} has been estimated in Section 6, which gives \begin{equation}\label{eq-J-trans-2}
\sum_{n\in\J} f(n)\tau(n-1)\ll X(\log X)^{\frac{1}{2}+\varepsilon}.
\end{equation}

Now we treat the second term on the second line of \eqref{eq-J-trans}. By the decomposition of $\J$ and Hypothesis (i), we have
\begin{equation}\label{eq-f-J}
\begin{aligned}
\sum_{n\in \J} |f(n)| &\ll \sum_{n\in\J_1\setminus\I} |f(n)| + \sum_{n\in\J_2\setminus\J_1} |f(n)| + \sum_{n\in\J_3} |f(n)| \\
&\ll X(\log X)^{\frac{c-1}{2}}(|\J_1\setminus\I|^{\frac{1}{2}} + |\J_2\setminus\J_1|^{\frac{1}{2}}) + \Big(\sum_{n\in\J_3}|f(n)|^2\Big)^{\frac{1}{2}} |\J_3|^{\frac{1}{2}}.
\end{aligned}
\end{equation}
Moreover, it follows from the definitions of these sets that
\begin{equation*}
\label{2-10}
|\J_1\setminus\I| \ll \sum_{0\leq \nu\leq L} \sum_{p\in \P_\nu} \frac{X}{H_{\nu}(\log X)^{A}}\ll  \frac{X}{(\log X)^{A-1}},
\end{equation*}
\begin{equation*}
|\J_2\setminus \J_1| \ll \sum_{0\leq \nu\leq L} \sum_{p_1,p_2\in\P_\nu} \frac{X}{p_1\,p_2}\ll  \frac{X}{(\log X)^{A-1}}.
\end{equation*}
and
\begin{equation*}
|\J_3|= \sum_{n\leq X \atop (n, P)=1} 1\ll \frac{X}{(\log X)^{1-\varepsilon}},
\end{equation*}
where $P=P\big(\exp\big((\log X)^{\varepsilon/2}\big), \, \exp\big((\log X)^{1-\varepsilon/2}\big)\big)$, and $A$ is an arbitrarily constant with $A>c+3$.  Inserting these estimates and Hypothesis (ii) into \eqref{eq-f-J}, we get
\[
\sum_{n\in \J} |f(n)| \ll \frac{X}{(\log X)^{1-\varepsilon}}.
\]
Substituting this and \eqref{eq-J-trans-2} into \eqref{eq-J-trans}, we finally derive
\begin{equation}\label{eq-J-trans-3}
\sum_{q\leq  Q}c_qE^{\J} (f, X; q)\ll X(\log X)^{\frac{1}{2}+\varepsilon}.
\end{equation}
This completes the proof of Theorem \ref{thm-BV}, in view of \eqref{eq-ap-I-F}.

\vskip 5mm

\section{Second proof of Theorem \ref{thm-main}}

We shall first deduce the following estimate from Section 8
\begin{equation}\label{eq-bv-q2}
\sum_{\substack{q \leq \sqrt{X}}} \Big ( \sum_{\substack{n \leq q^2 \\ n \equiv 1\moq}} f(n) - \frac{1}{\varphi(q)} \sum_{\substack{n \leq q^2 \\ (n,q) = 1}} f(n) \Big) \ll_{\varepsilon} X(\log X)^{\frac{1}{2} + \varepsilon}.
\end{equation}
With notation $\I, \J$ as in Section 8, it is easily seen from the argument in \eqref{eq-J-trans}-\eqref{eq-J-trans-3} that
\[
\sum_{\substack{q \leq \sqrt{X}}} \Big ( \sum_{\substack{n \leq q^2,\ n\in \J \\ n \equiv 1\moq}} f(n) - \frac{1}{\varphi(q)} \sum_{\substack{n \leq q^2,\  n\in \J \\ (n,q) = 1}} f(n) \Big) \ll_{\varepsilon} X(\log X)^{\frac{1}{2} + \varepsilon}.
\]
Then it suffices to estimate the contribution of terms from $n\in \I$.
Similar to the argument in Section 7, we shall decompose the sums over $q$ and $n$ into short intervals $\big(Q(1+\Delta_2)^{-1}, Q\big]$ and $\big(N(1+\Delta_2)^{-1}, N\big]$. Then we have
\begin{equation*}
\begin{aligned}
T^{\I}:=&\sum_{\substack{q \leq \sqrt{X}}} \Big ( \sum_{\substack{n \leq q^2,\  n\in \I \\ n \equiv 1\moq}} f(n) - \frac{1}{\varphi(q)} \sum_{\substack{n \leq q^2,\  n\in \I\\ (n,q) = 1}} f(n) \Big) \\
\ll &\sum_{j_1,j_2\geq 0}\bigg|\sum_{Q(1+\Delta_2)^{-1}<q\leq Q } \Big ( \sum_{\substack{N(1+\Delta_2)^{-1}<n\leq N\\ n \leq q^2,\  n\in \I \\ n \equiv 1\moq}} f(n) - \frac{1}{\varphi(q)} \sum_{\substack{N(1+\Delta_2)^{-1}<n\leq N\\ n \leq q^2,\  n\in \I\\ (n,q) = 1}} f(n) \Big) \bigg| \\
:=& \sum_{j_1,j_2\geq 0}|T^{\I}_{j_1,j_2}|,
\end{aligned}
\end{equation*}
where $Q=\sqrt{X}(1+\Delta_2)^{-j_1}, N=X(1+\Delta_2)^{-j_2}$. Since $n<N$ and $q>Q(1+\Delta_2)^{-1}$ in each sum $T^{\I}_{j_1,j_2}$, the condition $n \leq q^2$ can be dropped as soon as
\begin{equation}\label{eq-drop-1}
N\leq Q^2(1+\Delta_2)^{-2}.
\end{equation}
If the condition \eqref{eq-drop-1} is satisfied, the variables $q,n$ are independent. Thus, we can apply the estimate \eqref{eq-ap-I-F} for $N\geq X^{1-\varepsilon}$ and the large sieve inequality for $N< X^{1-\varepsilon}$. These give
\begin{equation*}\label{eq-Tij}
\sum_{\substack{j_1,j_2\geq 0\\ N\leq Q^2(1+\Delta_2)^{-2}}}|T^{\I}_{j_1,j_2}|\ll \Delta_2^{-2}(\log X)^{2}\frac{X}{(\log X)^{A}}
\end{equation*}
for any $A>0$. It remains to estimate the contribution of $j_1,j_2$ when \eqref{eq-drop-1} is not satisfied. The treatment here is different from the corresponding in Section 7, since there is no good upper bound for the individual  $f(n)$. From $n(1+\Delta_2)>N$ and $q\leq Q$, we deduce that $n>q^2(1+\Delta_2)^{-3}$. So the contribution of $j_1,j_2$ such that $N> Q^2(1+\Delta_2)^{-2}$ is at most
\begin{equation}\label{eq-ijE}
\begin{aligned}
&\sum_{q \leq \sqrt{X}} \Big ( \sum_{\substack{q^2(1+\Delta_1)^{-3}<n \leq q^2\\ n \equiv 1\moq}} |f(n)| +\frac{1}{\varphi(q)} \sum_{\substack{q^2(1+\Delta_1)^{-3}<n \leq q^2}} |f(n) |\Big) \\
\ll & X^{\frac{1}{2}}(\log X)^{\frac{c-1}{2}} \bigg(\sum_{n\leq X}\Big( \sum_{\substack{\sqrt{n}\leq q\leq \sqrt{n}(1+\Delta_1)^{3/2}\\ q|(n-1)}} 1\Big)^2\bigg)^{\frac{1}{2}}+ \Delta_1^{\frac{1}{2}} X(\log X)^{\frac{c-1}{2}},
\end{aligned}
\end{equation}
where we exchange the order of summations over $q$ and $n$, and use the Cauchy--Schwarz inequality and Hypothesis (i). We next use the trivial bound $\tau(n-1)$ for the sum over $q$ once and exchange the order of summations again. This gives
\begin{equation*}
\begin{aligned}
\sum_{n\leq X}\Big( \sum_{\substack{\sqrt{n}\leq q\leq \sqrt{n}(1+\Delta_1)^{3/2}\\ q|(n-1)}} 1\Big)^2\ll& \sum_{q \leq \sqrt{X}(1+\Delta_1)^{3/2}} \; \sum_{\substack{q^2(1+\Delta_1)^{-3}<n \leq q^2\\ n \equiv 1\moq}} \tau(n-1)\\
\ll& \sum_{q \leq \sqrt{X}(1+\Delta_1)^{3/2}}\tau(q)  \sum_{\substack{q(1+\Delta_1)^{-3}-1<n \leq q}} \tau(n)\\
\ll & \Delta_1 X(\log X)^{2}.
\end{aligned}
\end{equation*}
Inserting this into \eqref{eq-ijE}, we obtain that the contribution of $j_1,j_2$ such that $N\geq X^{1-\varepsilon}$ is bounded by $O\big(\Delta_1^{\frac{1}{2}} X(\log X)^{\frac{c+1}{2}}\big)$. On taking $\Delta_1=(\log X)^{-A/3}$ for some sufficiently large $A$, we get
\[
T^{\I}\ll \frac{X}{(\log X)^{A/4}},
\]
which further implies the estimate \eqref{eq-bv-q2}.

By the identity \eqref{eq-def-tau} and Hypothesis (i), we have
\[
\sum_{n \leq X}f(n)\tau(n-1)=2\sum_{q\leq \sqrt{X}}\sum_{\substack{q^2<n\leq X \\ n\equiv 1 \moq}} f(n)+O\big(X^{\frac{3}{4}+\varepsilon}\big).
\]
According to Theorem \ref{thm-BV} and the estimate \eqref{eq-bv-q2}, we get
\[
\sum_{n \leq X}f(n)\tau(n-1)=2\sum_{q\leq \sqrt{X}}\frac{1}{\varphi (q) } \sum_{\substack{q^2<n\leq X \\ (n,q) =1}} f(n)+O\big(X(\log X)^{\frac{1}{2} + \varepsilon}\big).
\]
Finally, Theorem \ref{thm-main} follows from exchanging the order of summations and inserting the estimate in Lemma \ref{lem-phisum}.


\begin{thebibliography}{10}

\bibitem{ABL-2021}
E. Assing, V. Blomer and J. Li,
\newblock Uniform {T}itchmarsh divisor problems.
\newblock {\em Adv. Math.}, 393:Paper No. 108076, 51,  2021.



\bibitem{BGHT}
T. Barnet-Lamb, D. Geraghty, M. Harris and R. Taylor.
\newblock A family of Calabi--Yau varieties and potential automorphy II.
\newblock {\em Publ. Res. Inst. Math. Sci.}, 47:29--98, 2011.

\bibitem{BC-2018}
S.~Bettin and V.~Chandee.
\newblock Trilinear forms with {K}loosterman fractions.
\newblock {\em Adv. Math.}, 328:1234--1262, 2018.



\bibitem{BSZ-2013}
J. Bourgain, P. Sarnak and T. Ziegler.
\newblock Disjointness of M\"obius from horocycle flows.
In {\it From Fourier Analysis and Number Theory to Radon Transforms and Geometry}, ed. by H.~M. Farkas {\it et al.}, Developments in Mathematics {\bf 28}, 67--83, Springer 2013.


\bibitem{CafPerZac20}
M. Cafferata, A. Perelli and A. Zaccagnini.
\newblock An extension of the Bourgain--Sarnak--Ziegler theorem with modular applications.
\newblock {\em Q. J. Math.}, 71:359--377, 2020.



\bibitem{D1974}
P. Deligne.
\newblock La conjecture de {W}eil. {I}.
\newblock {\em Publ. Math. Inst. Hautes \'{E}tudes Sci.}, 43:273--307, 1974.

\bibitem{Drappeau-2017}
S. Drappeau.
\newblock Sums of Kloosterman sums in arithmetic progressions, and the error term in the dispersion method.
\newblock {\em  Proc. Lond. Math. Soc.}, 114(3):684--732, 2017.


\bibitem{DT-2019}
S. Drappeau and B. Topacogullari.
\newblock Combinatorial identities and {T}itchmarsh's divisor problem for
multiplicative functions.
\newblock {\em Algebra Number Theory}, 13(10):2383--2425, 2019.



\bibitem{FR-2022}
{\'E}. Fouvry and M. Radziwi\l\l.
\newblock Level of distribution of unbalanced convolutions.
\newblock {\em Ann. Sci. \'{E}c. Norm. Sup\'{e}r. (4)}, to appear.

\bibitem{FoTen-1985}
\'E. Fouvry and G. Tenenbaum.
\newblock  Sur la corr\'{e}lation des fonctions de {P}iltz.
\newblock {\em Rev. Mat. Iberoamericana}, 1(3):43--54, 1985.



\bibitem{FoTen-2022}
\'E. Fouvry and G. Tenenbaum.
\newblock Multiplicative functions in large arithmetic progressions and applications.
\newblock {\em Trans. Amer. Math. Soc.}, 375(1):245--299, 2022.




\bibitem{GJ-1978}
S. Gelbart, H. Jacquet.
\newblock A relation between automorphic representations of ${\rm
		GL}(2)$ and ${\rm GL}(3)$,
\newblock {\em Ann. Sci. \'{E}cole Norm. Sup. (4)}, 11:471--542, 1978.
	
\bibitem{GS-2019}
	A. Granville and X. Shao.
	\newblock Bombieri-{V}inogradov for multiplicative functions, and beyond the
	{$x^{1/2}$}-barrier.
	\newblock {\em Adv. Math.}, 350:304--358, 2019.
	
\bibitem{Green}
B.~Green.
\newblock A note on multiplicative functions on progressions to large moduli.
\newblock {\em Proc. Roy. Soc. Edinburgh Sect. A}, 148(1):63--77, 2018.


\bibitem{ht}
R.~R. Hall and G. Tenenbaum.
\newblock  Divisors.
\newblock Cambridge Tracts in Mathematics, 90. Cambridge
University Press, Cambridge, 1988.

\bibitem{Heath-Brown-1986}
D.~R. Heath-Brown.
\newblock The divisor function $d_3(n)$ in arithmetic progressions.
\newblock {\em Acta Arith.}, 47:29–56, 1986.




\bibitem{IK}
H. Iwaniec and E. Kowalski.
\newblock {\em Analytic number theory}, vol.~53 of {\em American Mathematical
	Society Colloquium Publications}.
\newblock American Mathematical Society, Providence, RI, 2004.



\bibitem{JL2019}
Y. Jiang and G. L\"u.
\newblock The generalized Bourgain-Sarnak-Ziegler criterion and its application to additively twisted sums on ${\rm GL}_m$.
\newblock {\em Sci. China Math.}, 64:2207--2230, 2021.


\bibitem{JL-2021}
Y. Jiang and G. L\"u.
\newblock Cancellation in algebraic twisted sums on ${\rm GL}_m$.
\newblock {\em Forum Math.}, 33:1061--1082, 2021.

\bibitem{JL-2021-submit}
Y. Jiang and G. L\"u.
\newblock Summation formulae of arithmetic functions over arithmetic progressions and applications.
Submitted

\bibitem{JLW}
Y. Jiang, G. L\"u and Z. Wang.
\newblock Exponential Sums with Multiplicative Coefficients without the Ramanujan conjecture.
\newblock {\em Math. Ann.}, 379:589--632, 2021.

\bibitem{JLW-mobius}
\newblock  Y. Jiang, G. L\"u and Z. Wang.
\newblock M\"obius randomness law for ${\rm GL}(m)$ automorphic $L$-functions twisted by additive characters.
Submitted


\bibitem{JLTW-2021}
Y. Jiang, G. L\"u, J. Thorner and Z. Wang.
\newblock A Bombieri--Vinogradov theorem for higher-rank groups.
 \newblock{\em Int. Math. Res. Not.}, doi:10.1093/imrn/rnab261, 2021.

\bibitem{Katai-1986}
I. K\'{a}tai.
\newblock A remark on a theorem of {H}. {D}aboussi.
 \newblock{\em Acta Math. Hungar.}, 47 (1-2):223--225, 1986.


\bibitem{Kim-2003}
H.~H. Kim.
\newblock  Functoriality for the exterior square of ${\rm GL}_4$ and the symmetric fourth of ${\rm GL}_2$,
\newblock {\em J. Amer. Math. Soc.}, 16:139--183, 2003.

\bibitem{KS-2002}
H.~H. Kim and F. Shahidi.
\newblock Functorial products for {${\rm GL}_2\times{\rm GL}_3$} and the
symmetric cube for {${\rm GL}_2$}. With an appendix by Colin J. Bushnell and Guy Henniart,
\newblock {\em Ann. of Math. (2)}, 155:837--893, 2002.

\bibitem{LauLu-2011}
 Y.-K. Lau and G. L\"u.
\newblock  Sums of Fourier coefficients of cusp forms. (English summary)
\newblock {\em Q. J. Math.}, 62(3):687--716, 2011.

\bibitem{Linnik-1963}
Ju. V. Linnik.
\newblock {\em The Dispersion Method in Binary Additive Problems}, Translated by S. Schuur, American Mathematical Society, Providence, RI, 1963.




\bibitem{lu-2015}
G. L\"u.
\newblock Shifted convolution sums of Fourier coefficients with divisor functions.
\newblock {\em Acta Math. Hungar.}, 146:86--97, 2015.

\bibitem{MRT-2019}
K. Matomaki, M. Radziwi\l\l and T. Tao.
\newblock Correlations of the von mangoldt and higher divisor functions ii: Divisor correlations in short ranges.
\newblock {\em Math. Ann.}, 374(1-2):793--840, 2019.



\bibitem{Motohashi-1980}
Y. Motohashi.
\newblock An asymptotic series for an additive divisor problem.
\newblock {\em Math. Z.}, 170:43--63, 1980.


\bibitem{Motohashi-1994}
Y. Motohashi.
\newblock The binary additive divisor problem.
\newblock {\em Ann. Sci. \'{E}cole Norm. Sup. (4)}, 27(5):529--572, 1994.





\bibitem{PonVaughan-2015}
P. Pongsriiam and R.Vaughan.
\newblock   The divisor function on residue classes I.
\newblock  {\em Acta Arith.}, 168(4):369--382, 2015.


\bibitem{NT2019}
J. Newton and J. A. Thorne.
\newblock Symmetric power functoriality for holomorphic modular forms.
\newblock {\em Publ. Math. Inst. Hautes \'{E}tudes Sci.}, 134:1--116, 2021.


\bibitem{RE1}
D. Redmond.
\newblock {An asymptotic formula in the theory of numbers}.
\newblock {\em Math. Ann.}, 224:247--268, 1976.

\bibitem{RE2}
D. Redmond.
\newblock {An asymptotic formula in the theory of numbers. II}.
\newblock {\em Math. Ann.}, 234:221--238, 1978.

\bibitem{RE3}
D. Redmond.
\newblock {An asymptotic formula in the theory of numbers. III}.
\newblock {\em Math. Ann.}, 243:143--151, 1979.


\bibitem{RS}
Z. Rudnick and P. Sarnak.
\newblock Zeros of principal {$L$}-functions and random matrix theory.
\newblock {\em Duke Math. J.}, 81:269--322, 1996.


\bibitem{Shahidi-1981}
F. Shahidi.
\newblock  On certain $L$-functions.
\newblock {\em Amer. J. Math.}, 103:297--355, 1981.




\bibitem{Shiu}
P. Shiu.
\newblock A Brun--Titchmarsh theorem for muliplicative functions.
\newblock{\em J. Reine Angew. Math.}, 313:161--170, 1980.

\bibitem{Smith-1982}
R. A. Smith.
\newblock The generalized divisor problem over arithmetic progressions.
\newblock{\em Math. Ann.}, 260:255--268, 1982.

\bibitem{TW-2016}
H. Tang and J. Wu.
\newblock Fourier coefficients of symmetric power $L$-functions.
\newblock {\em J. Number Theory}, 167:147--160, 2016.


\bibitem{TT-1998}
N. M. Timofeev and S.~T. Tulyaganov.
\newblock Problems similar to the additive divisor problem.
\newblock{\em Mat. Zametki [Math. Notes]}, 64(3):443--456, 1998.


\bibitem{Topacogullari-2016}
B. Topacogullari.
\newblock The shifted convolution of divisor functions.
\newblock{\em Q. J. Math.}, 67(2):331–363, 2016.


\bibitem{Topacogullari-2018}
B. Topacogullari.
\newblock The shifted convolution of generalized divisor functions.
\newblock{\em Int. Math. Res. Not.}, 24:7681--7724, 2018.

\bibitem{Wolke-1973}
D. Wolke.
\newblock \" {U}ber die mittlere {V}erteilung der {W}erte zahlentheoretischer
{F}unktionen auf {R}estklassen. {I}.
\newblock {\em Math. Ann.}, 202:1--25, 1973.

\end{thebibliography}
\end{document}